\documentclass[11pt,oneside]{amsart}
\usepackage{amsmath,ifthen, amsfonts, amssymb, srcltx,
   amsopn, textcomp}

\usepackage{hyperref}
\usepackage{graphicx}

\usepackage[small]{caption}

\usepackage{tikz}
\usetikzlibrary{matrix}
\usetikzlibrary{decorations.pathreplacing}

\usepackage{geometry}
\usepackage{graphics}

\usepackage{mathrsfs}
\usepackage{stmaryrd}
\usepackage{tikz}
\usepackage{tikz-cd}
\usepackage{framed}
\usepackage{mathtools}
\usepackage{xfrac}
\usepackage{faktor}
\usepackage{stackengine}
\usepackage{braket}
\usepackage{comment}
\usepackage{cancel}
\usepackage{xcolor}

\usepackage{mathrsfs}

\usepackage{amssymb}

\tikzset{
math to/.tip={Glyph[glyph math command=rightarrow]},
loop/.tip={Glyph[glyph math command=looparrowleft, swap]},
loop'/.tip={Glyph[glyph math command=looparrowleft]},
 weird/.tip={Glyph[glyph math command=Rrightarrow, glyph length=1.5ex]},
  pi/.tip={Glyph[glyph math command=pi, glyph length=1.5ex, glyph axis=0pt]},
}

\usepackage{stackengine}

\newcommand{\showcomments}{yes}
\renewcommand{\showcomments}{no}

\newcommand{\hidetodo}[1]
{\ifthenelse{\equal{\showcomments}{yes}}%
{#1}
}

\newsavebox{\commentbox}
\newenvironment{com}%
{\ifthenelse{\equal{\showcomments}{yes}}%
{\footnotemark
        \begin{lrbox}{\commentbox}
        \begin{minipage}[t]{1.25in}\raggedright\sffamily\tiny
        \footnotemark[\arabic{footnote}]}
{\begin{lrbox}{\commentbox}}}%
{\ifthenelse{\equal{\showcomments}{yes}}%
{\end{minipage}\end{lrbox}\marginpar{\usebox{\commentbox}}}
{\end{lrbox}}}

\newtheorem{thm}{Theorem}[section]
\newtheorem{lem}[thm]{Lemma}

\newtheorem{cor}[thm]{Corollary}

\newtheorem{prop}[thm]{Proposition}

\theoremstyle{definition}
\newtheorem{defn}[thm]{Definition}
\newtheorem{rem}[thm]{Remark}

\DeclareMathOperator{\rank}{rank}
\DeclareMathOperator{\mrank}{\widetilde{rank}}

\DeclareMathOperator{\link}{link}

\DeclareMathOperator{\stab}{Stab}

\DeclareMathOperator{\id}{Id}

\begin{document}

\title[Generalized Echelon Subgroups]{Generalized Echelon Subgroups}

\author[B.~Abdenbi]{Brahim Abdenbi}
\email{brahim.abdenbi@mail.mcgill.ca}

          \address{Dept. of Math. \& Stats.\\
                    McGill Univ. \\
                    Montreal, Quebec, Canada H3A 0B9}
\thanks{Research supported by NSERC}
\date{\today}

\maketitle

\begin{com}
{\bf \normalsize COMMENTS\\}
ARE\\
SHOWING!\\
\end{com}

\begin{abstract}
{A subgroup $\mathcal{H}$ of a free group $\mathcal{F}$ is \textit{inert} if for any subgroup $\mathcal{K}\subset \mathcal{F}$, we have $\rank\left(\mathcal{H}\cap \mathcal{K}\right)\leq \rank\left(\mathcal{K}\right)$. It is \textit{compressed} if $\rank\left(\mathcal{H}\right)\leq \rank\left(\mathcal{K}\right)$ whenever $\mathcal{H}\leq \mathcal{K}$. In this paper, we introduce \textit{highly inert} graph immersions and show that they represent inert subgroups. We use the compressibility of inert subgroups to prove new properties on label distributions in their corresponding graphs. Our main result is the generalization of Rosenmann's \textit{echelon subgroups}, which he showed to be inert using endomorphisms of free groups. We show that the collection of echelon subgroups is a proper sub-collection of \textit{generalized echelon} subgroups. Using some techniques from Mineyev-Dicks proof of the Hanna Neumann Conjecture, we show inertness of generalized echelon subgroups, thus providing a new proof for inertness of echelon subgroups.}
 \end{abstract}

\section{Introduction}
\noindent The notion of \textit{inert} subgroups was first introduced in $1996$ by Dicks and Ventura in \cite{MR1385923}. A subgroup $\mathcal{H}$ of a free group $\mathcal{F}$ is \textit{inert} if for any subgroup $\mathcal{K}\leq \mathcal{F}$, the rank of $\mathcal{H}\cap \mathcal{K}$ is bounded from above by the rank of $\mathcal{K}$. Such subgroups arise as fixed subgroups of of injective endomorphisms $\mathcal{F}\rightarrow \mathcal{F}$. Although this area of research was largely motivated by the work of Dyer and Scott \cite{MR369529} in $1975$, interest in subgroups of free groups and their ranks dates back to earlier works by many including Nielsen, Schreier, and Howson.

In $1926$, Nielsen and Schreier  proved that subgroups of free groups are free \cite{MR0422434} \cite{MR695161}. In the case where the subgroup is of finite index, they gave an explicit formula for computing its rank, namely the Nielsen-Schreier formula. In $1954$, Howson \cite{Howson54}  showed that the intersection of finitely generated subgroups is finitely generated. In particular, he showed that if $\mathcal{H}$ and  $\mathcal{K}$ are subgroups of finite ranks, $\rank\left(\mathcal{H}\right)$ and  $\rank\left(\mathcal{K}\right)$ respectively, then the rank of their intersection is bounded above by $2\rank\left(\mathcal{H}\right)\rank\left(\mathcal{K}\right)-\rank\left(\mathcal{H}\right)-\rank\left(\mathcal{K}\right)+1$. Soon afterward, Hanna Neumann \cite{MR1092229} improved on Howson's bound by showing that the rank of the intersection is bounded from above by $2\left(\rank\left(\mathcal{H}\right)-1\right)\left(\rank\left(\mathcal{K}\right)-1\right)+1$, and further conjectured an even lower bound of $\left(\rank\left(\mathcal{H}\right)-1\right)\left(\rank\left(\mathcal{K}\right)-1\right)+1$, which came to be known as the Hanna Neumann Conjecture. This conjecture was solved by Friedman \cite{MR3289057} in $2011$ and independently by I. Mineyev \cite{MR2914871} in the same year.\\
In $1996$ Dicks and Ventura \cite{MR1385923} introduced the notion of inert subgroups and showed that fixed subgroups of injective endomorphisms of free groups are inert. In the same article they introduced the notion of \textit{compressed} subgroups. A subgroup $\mathcal{H}\leq \mathcal{F}$ is compressed if for any subgroup $\mathcal{K}\leq \mathcal{F}$, if $\mathcal{H}\leq \mathcal{K}$ then $\rank\left(\mathcal{H}\right)\leq \rank\left(\mathcal{K}\right)$. Inert subgroups are compressed. However, it remains an open question whether compressed subgroups are inert or not. In \cite{MR3245107}, Rosenmann introduced \textit{echelon} subgroups and showed that they are inert. These subgroups arise as images of special endomorphisms called \textit{$1$-generator endomorphisms}. The main result of this paper is the generalization of echelon subgroups.\\
Inert subgroups are mainly studied using endomorphisms of free groups, since they first arose in \cite{MR1385923} as fixed subgroups of injective endomorphisms of free groups. Our approach, however, is mostly graph theoretic. In Section~\ref{sec:intro}, we establish notations, and recall some classical definitions and theorems regarding graphs and free groups. In Section~\ref{sec:inert and compressed}, we introduce the notion of \textit{highly inert} immersions which will allow us to construct examples of inert subgroups. We also show some properties of compressed graphs pertaining to the distribution of labels in their edge sets. The main result in this section is the proof that compressed graphs admit maximal essential sets that map injectively into the bouquet of circles. In Section~\ref{sec:ech} we introduce the class of \textit{generalized echelon} subgroups, and show that it properly contains Rosenmann's echelon subgroups. We give a short overview of Mineyev-Dicks' proof of the Hanna Neumann conjecture, and use some of its results to show inertness of generalized echelon subgroups, thus providing a new proof for inertness of echelon subgroups.
\section{Preleminaries}\label{sec:intro}
\subsection{Graphs and Morphisms}
\noindent A \textit{directed graph} $\Gamma$ is a $1$-dimensional $CW$-complex. The sets of its \textit{vertices} and \textit{edges}, denoted by  $\Gamma^{0}$ and  $\Gamma^1$, are the $0$-cells and open $1$-cells, respectively. There exist two \textit{incidence} maps $o, \tau : \Gamma^1\rightarrow \Gamma^0$ mapping each edge $e\in \Gamma^1$ to its \textit{boundary vertices}, $o\left(e\right),\ \tau\left(e\right)$ which we refer to as the \textit{origin} and \textit{terminus} of $e$, respectively. The edge $e$ is oriented from $o\left(e\right)$ to $\tau\left(e\right)$. A \textit{morphism} of graphs $\phi:\Gamma_1\rightarrow \Gamma_2$ is a continuous map that sends vertices to vertices and edges to edges homeomorphically. If a \textit{base} vertex $v$ is chosen in $\Gamma_1$, then $\phi: \left(\Gamma_1,v\right)\rightarrow \left(\Gamma_2,\phi\left(v\right)\right)$ is a \textit{based morphism}. A \textit{bouquet of $n$ circles} is a graph $B$ with a single vertex and $n$ edges. For simplicity, a based morphism into $B$ is denoted by $\left(\Gamma,v\right)\rightarrow B$. A \textit{labelling} of a graph $\Gamma$ is a morphism $\ell:\Gamma\rightarrow B$. A morphism $\phi: \Gamma_1\rightarrow \Gamma_2$ is \textit{label preserving} if the following diagram commutes
\[
\begin{tikzcd}
\Gamma_1 \arrow[r, "\phi"] \arrow[dr, "\ell_1"]
& \Gamma_2 \arrow[d,"\ell_2"]\\
& B
\end{tikzcd}
\]
A morphism is an \textit{immersion} if it is locally injective. Unless otherwise specified, all graphs $\Gamma$ are compact and equipped with a fixed labelling $\ell:\Gamma\rightarrow B$, where $\ell$ is an immersion.\\

\noindent Given a non-negative integer $m$,  we denote by $I_m$ the graph homeomorphic to the interval $\left[0,m\right]\subset \mathbb{R}$ where $I_m^0=\left[0,m\right]\cap \mathbb{Z}$ and $I_m^1=\left\{\left(i,i+1\right)\mid 0\leq i\leq m-1\right\}$. A \textit{path} $p$ of length $m$ joining two vertices $v$ and $w$ is a morphism $p: I_m\rightarrow \Gamma$ such that $p\left(0\right)=v$ and $p\left(m\right)=w$. When $v=w$, $p$ is a \textit{closed path}. In particular, $p$ is a \textit{cycle} if it is closed and injective on $\left(0,m\right)$. A cycle of length $1$ is a \textit{loop}. A \textit{concatenation} of two paths $p: I_m\rightarrow \Gamma$ and $p': I_{m'}\rightarrow \Gamma$ is a path $\gamma: I_{m+m'}\rightarrow \Gamma$ such that $\gamma|_{\left[0,m\right]}=p$ and $\gamma|_{\left[m,m+m'\right]}=p'$. Note that this requires $p\left(m\right)=p'\left(0\right)$.\\

\noindent A \textit{subgraph} is a subcomplex. A graph is \textit{connected} if any two of its vertices can be joined by a path. A \textit{(connected) component} is a maximal connected subgraph with respect to inclusion. A \textit{forest} is a graph that contains no cycles and a connected component of a forest is a \textit{tree}.\\ 
The \textit{link} of a vertex $v\in \Gamma^0$, denoted by $\link\left(v\right)$, is the set of all length $1$ paths starting at $v$
$$\link\left(v\right)=\left\{p: I_1\rightarrow \Gamma \mid p\left(0\right)=v\right\}$$
A loop at $v$ contributes two paths to $\link\left(v\right)$. Observe that an immersion $\phi:\Gamma_1\rightarrow \Gamma_2$ induces injective maps $\link\left(v\right)\rightarrow \link\left(\phi\left(v\right)\right)$ for all $v\in\Gamma_1^0$.\\

\noindent The \textit{degree} of a vertex $v$, denoted $\deg\left(v\right)$, is the cardinality of $\link\left(v\right)$. If $|\link\left(v\right)|<\infty$ for all $v\in \Gamma^0$ then $\Gamma$ is \textit{locally finite}. In this paper, we only consider graphs whose vertices have uniformly bounded degrees, ie there exists $D\geq 0$ such that $|\link\left(v\right)|\leq D$ for all $v\in \Gamma^0$. A finite graph is a \textit{core} if all its vertices have degrees greater than or equal to $2$. Any graph that is not a forest deformation retracts to a core subgraph. A vertex $v$ is a \textit{branching} vertex if $\deg\left(v\right)\geq 3$. Given a set of vertices $S$, we define $S^*$ as $$S^*=\left\{v\in S\mid \deg\left(v\right)\geq 3\right\}.$$In particular, $\Gamma^*$ is the set of branching vertices of $\Gamma$.

\noindent The \textit{Euler characteristic} of a compact graph $\Gamma$ is $\chi\left(\Gamma\right)=\left|\Gamma^0\right|-\left|\Gamma^1\right|$. Its \textit{reduced rank}, denoted by $\mrank\left(\Gamma\right)$, is $\displaystyle\sum_{\Gamma_i\subset \Gamma} \max\left\{0,-\chi\left(\Gamma_i\right)\right\}$, where $\Gamma_i$ are the components of $\Gamma$. The \textit{rank} of a component $\Gamma_i$ is $\rank\left(\Gamma_i\right)=1-\chi\left(\Gamma_i\right)$. Observe that finite connected graphs have the same rank as their core subgraphs.\\
We now give a few results whose proofs are omitted but can be found in \cite{MR695906}.
\begin{lem} \label{imm}
Composition of immersions is an immersion.
\end{lem}
\begin{lem}
Given an immersion of graphs $\phi: \Gamma_1 \rightarrow \Gamma_2$ and a vertex $v\in \Gamma^0$, the induced homomorphism of fundamental groups $$\phi_*: \pi_1\left(\Gamma_1,v\right)\rightarrow \pi_1\left(\Gamma_2, \phi\left(v\right)\right)$$
is injective.
\end{lem}
\subsection{Foldings}
\noindent If $\ell:\Gamma\rightarrow B$ is not an immersion, then we call a pair of edges $\left(e_1, e_2\right)\in \Gamma^1\times \Gamma^1$ \textit{admissible} if
\begin{enumerate}
\item $o\left(e_1\right)=o\left(e_2\right)$ or $\tau\left(e_1\right)=\tau\left(e_2\right)$, and
\item $\ell\left(e_1\right)=\ell\left(e_2\right)$.
\end{enumerate}
A \textit{folding} of $\Gamma$ is a map $$f:\Gamma\rightarrow \Gamma/\left(e_1\sim e_2\right)$$
realized by identifying $e_1$ and $e_2$. For example, if $\ell\left(e_1\right)=\ell\left(e_2\right)=a$, then a folding is the procedure shown in Figure~\ref{fig:graph_6}.
\begin{figure}[t]\centering
\includegraphics[width=.4\textwidth]{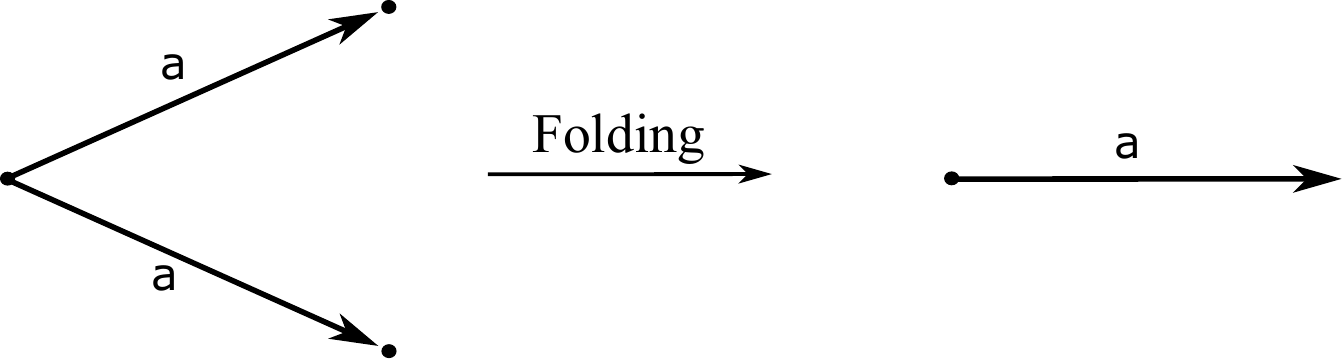}
\caption[Folding]{\label{fig:graph_6}
Folding of two edges}
\end{figure}
\begin{lem}\label{lem:folding}
If  $\left(e_1, e_2\right)\in \Gamma^1\times \Gamma^1$ is an admissible pair, then the folding $$f:\Gamma\rightarrow \Gamma/\left(e_1\sim e_2\right)$$ is $\pi_1$-surjective. Therefore, finite compositions of foldings are also $\pi_1$-surjective.
\end{lem} 
\noindent In particular, if $\Gamma$ is connected, then $\rank\left(\Gamma\right)\geq \rank\left(\Gamma/\left(e_1\sim e_2\right)\right)$.
\subsection{Graphs and Subgroups of Free Groups}\label{sec:graphs}
\noindent Let $\mathcal{F}=\pi_1B$. Finitely generated subgroups of $\mathcal{F}$ can be represented  by immersions of finite graphs into $B$. We summarize the algorithmic construction of such graphs below and refer the reader to 
\cite{MR695906} \cite{MR1882114}
 for more details.\\

\noindent Let $\mathcal{H}=\left<h_1,\ldots,h_r\right>\leq \mathcal{F}$ be a finitely generated subgroup. We start by first taking a disjoint union of $r$ circles such that for each $1\leq i\leq r$, the circle $C_i$ is subdivided to form a closed labelled cycle which reads as the generator $h_i$ starting from some fixed vertex $v_i$. We construct the based graph $$\left(H,v\right)=\frac{\displaystyle\bigsqcup_{i=1}^rC_i}{v_1\sim v_2\sim\ldots\sim v_r}$$
where all chosen vertices are identified to a single vertex $v$. If $\left(H,v\right)$ immerses into $B$ then we are done, otherwise perform all possible foldings until $\left(H,v\right)\rightarrow B$ is an immersion. The resulting graph is the desired one. In particular, $\pi_1\left(H,v\right)=\mathcal{H}$.\\
Henceforth, we will denote groups with script letters and their corresponding graphs with regular capital letters; for example the subgroup $\mathcal{H}$ is represented by the graph $H$. 
\begin{rem}
If the generators $h_i$ are cyclically reduced then all vertices of $H$ have degrees $\geq 2$ except for possibly the base vertex. 
\end{rem}

\subsection{Fiber Product of Immersions}
\noindent Let $\phi: H\rightarrow \Gamma$ and $\psi: K \rightarrow \Gamma$ be two immersions of finite graphs. The \textit{pullback} of these two maps 
\[
\begin{tikzcd}
A \arrow[r,"\alpha"] \arrow[d,black,"\beta"]
& K \arrow[d,"\psi"] \\
H \arrow[r, "\phi"]
& \Gamma
\end{tikzcd}
\]
also called the \textit{fiber product} is the graph $$A=H\otimes_{\Gamma}K$$
defined as follows
\begin{enumerate}
\item $A^0=H^0\times K^0$
\item $A^1=\left\{\left(e_1,e_2\right)\in H^1\times K^1\mid \phi\left(e_1\right)=\psi\left(e_2\right)\right\}$ 
\end{enumerate}
The immersions $\alpha: A\rightarrow K,\  \beta:A\rightarrow H$ are \textit{projections}.\\
We are mainly interested in the case where $\Gamma=B$ and $\phi,\ \psi$ are immersions. 
\begin{thm}\label{thm:fiber product thm}
Let 
\[
\begin{tikzcd}
A \arrow[r,"\alpha"] \arrow[d,black, "\beta"]
& K \arrow[d,"\ell_K"] \\
H \arrow[r, "\ell_H"]
& B
\end{tikzcd}
\]
be a pullback diagram of graph immersions $\ell_H$ and $\ell_K$. Let $v'=\left(u,v\right)\in A^0$ be such that $\alpha\left(v'\right)=v$ and $\beta\left(v'\right)=u$.  Define $F=\ell_K\circ \alpha=\ell_H\circ \beta$. 
Then $$F_*\left(\pi_1\left(A,v'\right)\right)=\ell_{K^*}\left(\pi_1\left(H,u\right)\right)\cap \ell_{H^*}\left(\pi_1\left(K,v\right)\right)$$
\end{thm}
\noindent In other words, the based component of the fiber product is precisely the graph representing the intersection of the fundamental groups of $K$ and $H$. Note that the three subgroups all lie inside $\mathcal{F}$. 
\begin{rem}\label{rem:cardinality of fibers}
The cardinality of intersections of fibers is at most $1$:
$$|\alpha^{-1}\left(v\right)\cap \beta^{-1}\left(u\right)|\leq 1\hspace{5mm} \text{and}\hspace{5mm} |\alpha^{-1}\left(e_2\right)\cap \beta^{-1}\left(e_1\right)|\leq 1$$

\noindent That is, if two vertices (edges) in $A$ project to the same vertex (edge) in $K$ then they must project to distinct vertices (edges) in $H$. This is because both $\ell_K\circ \alpha: A\rightarrow B$ and $\ell_H\circ \beta: A\rightarrow B$ are immersions.
\end{rem}
\noindent Two immediate corollaries to this theorem are
\begin{cor}
The intersection of finitely generated subgroups is itself finitely generated. .
\end{cor}
\noindent This result was first obtained by Howson in \cite{Howson54}.
\begin{cor}
 If subgroups $\mathcal{H}, \mathcal{K}\leq \mathcal{F}$ are finitely generated then as $w$ varies over $\mathcal{F}$, the subgroups  $\mathcal{K}\cap w^{-1}\mathcal{H}w$ belong to only a finite number of conjugacy classes of $\mathcal{F}$.  
\end{cor}
\subsection{The Combinatorial Gauss-Bonnet Theorem for Graphs}\label{subsection:GBT}
\noindent Given a graph $\Gamma$, one can define the curvature at a vertex $v\in \Gamma^0$ according to the following formula:
$$\kappa\left(v\right)=\pi\left(2-deg\left(v\right)\right)$$
When $\Gamma$ is finite, the \textit{Combinatorial Gauss-Bonnet} theorem relates the Euler characteristic of $\Gamma$ to the curvature in the following way
\begin{thm}\label{thm:GBT} 
$$2\pi\chi\left(\Gamma\right)=\displaystyle\sum_{v\in \Gamma^0}\kappa\left(v\right)$$
\end{thm}
\newpage

\section{Inert and Compressed Subgroups}\label{sec:inert and compressed}
\noindent The \textit{rank} of a group $\mathcal{G}$, denoted by $\rank\left(\mathcal{G}\right)$, is:
$$\rank\left(\mathcal{G}\right)=\inf\left\{\ \left|S\right|\ :\ \mathcal{G}=\left<S\right>\ \right\}$$
When $\mathcal{G}$ is free, its rank coincides with the rank of the graph representing it. 
\subsection{Inert Subgroups}
\noindent The following definition is due to Dicks-Ventura \cite{MR1385923}.
\begin{defn}\label{defn:inert subgroup}
A subgroup $\mathcal{H}\leq \mathcal{F}$ is \textit{inert} if for any subgroup $\mathcal{K}\leq \mathcal{F}$ we have: $$\rank\left(\mathcal{H}\cap \mathcal{K}\right)\leq \rank\left(\mathcal{K}\right).$$
\end{defn}
\noindent This can be formulated in terms of graphs as follows: an immersed based graph $\left(H,v\right)\rightarrow B$ is \textit{inert} if given any immersed based graph $\left(K,w\right)\rightarrow B$ we have $$\rank\left(H\otimes_{B}K,\left(v,w\right)\right)\leq \rank\left(K,w\right).$$
Henceforth, all graphs representing subgroups are understood to be based graphs. Thus for example, instead of writing $\left(H,v\right)$ we will simply write $H$.
\begin{lem}\label{lem:inert iff}
Let $H\rightarrow B$ be an immersion representing the subgroup $\mathcal{H}$. Then
$$H\rightarrow B\ \text{inert}\ \iff\ \mathcal{H}\ \text{inert}$$
\end{lem} 
\begin{proof}
$\left(\Rightarrow\right)$ Let $\mathcal{K}$ be a subgroup of $\mathcal{F}$ and $K\rightarrow B$ be the immersed graph representing it. By Theorem~\ref{thm:fiber product thm}, $\pi_1\left(H\otimes_BK\right)=\mathcal{H}\cap\mathcal{K}$ where $H\otimes_BK\rightarrow B$ is the immersed based component of the fiber product. Since $H\rightarrow B$ is inert, we have
$$\rank\left(\mathcal{H}\cap\mathcal{K}\right)=\rank\left(H\otimes_BK\right)\ \leq\ \rank\left(K\right)=\rank\left(\mathcal{K}\right).$$
$\left(\Leftarrow\right)$ Similarly, let $K\rightarrow B$ be an immersed graph. Then
\begin{equation*}
\rank\left(H\otimes_BK\right)=\rank\left(\mathcal{H}\cap\mathcal{K}\right)\ \leq\ \rank\left(\mathcal{K}\right)=\rank\left(K\right). \qedhere
\end{equation*}
\end{proof}
\begin{defn}\label{defn:inert set}
Let $\ell:H\rightarrow B$ be an immersed graph and $S\subset H^0$. Then $S$ is \textit{inert} if for any branching vertex $v\in S^*$ we have 
$$\displaystyle\sum_{w\in \left(S^*-\left\{v\right\}\right)}\left|\ell\left(\link\left(v\right)\right)\cap\ell\left(\link\left(w\right)\right)\right|\leq 2$$
This says that the directed edges at each branching vertex appear at most twice (with multiplicity) among the remaining branching vertices. See Figure~\ref{fig:graph_8}\\
\end{defn}
\begin{figure}[t]\centering
\includegraphics[width=.4\textwidth]{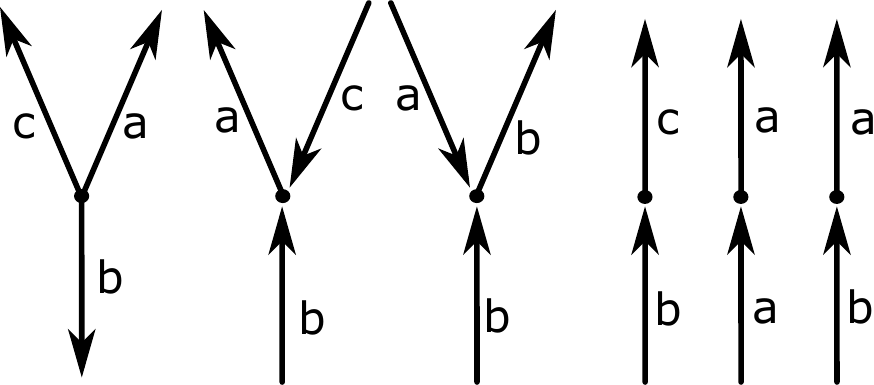}
\caption[Example of an inert set of vertices.]{\label{fig:graph_8}
Example of an inert set of vertices.}
\end{figure}
\begin{defn}\label{defn:highly inert}
An immersion of graphs $\phi: H\rightarrow K$ is \textit{highly inert} if for any vertex $w\in K^0$ the set $\phi^{-1}\left(w\right)$ is inert. In particular, $\phi: H\rightarrow B$ is highly inert if and only if $H^0$ is inert.
\end{defn}
\begin{lem}\label{lem:composition of highly inert}
Consider the following commutative diagram of graph immersions:
\[
\begin{tikzcd}[row sep=tiny]
& H \arrow[dd] \\
A \arrow[ur,"\beta"] \arrow[dr] & \\
& B
\end{tikzcd}
\]
Let $S\subset A^0$ be such that $\beta|_S$ is injective. Then
$$\beta\left(S\right)\ \text{inert}\ \Rightarrow\ S\ \text{inert}.$$
\end{lem}
\begin{proof}
The restriction of $\beta$ to the links of $S$ is an injective map $$\displaystyle\bigsqcup_{w\in S}\link\left(w\right)\hookrightarrow \displaystyle\bigsqcup_{v\in \beta\left(S\right)}\link\left(v\right).$$
Then $\beta\left(\link\left(w\right)\right)\subset \link\left(\beta\left(w\right)\right)$ for all $w\in S$. Let $v_i=\beta\left(w_i\right)$, where \\$S^*=\left\{w_1,\ldots,w_m\right\}$. Then, 
\begin{equation}\label{eq:link eq}
\left|\ell\left(link\left(v_i\right)\right)\cap\ell\left(\link\left(v_j\right)\right)\right|\ \geq\ \left|\ell\left(link\left(w_i\right)\right)\cap\ell\left(\link\left(w_j\right)\right)\right|.
\end{equation}
Inequality~\ref{eq:link eq} follows from the fact that if $X_1, X_2$ are sets such that $X_1'\subset X_1$ and $X_2'\subset X_2$, then $\left|X_1\cap X_2\right|\ \geq\ \left|X_1'\cap X_2'\right|$.\\ 
Now suppose there exists an index $1\leq i\leq m$ for which $$\displaystyle\sum_{i\neq j}\left|\ell\left(\link\left(w_i\right)\right)\cap\ell\left(\link\left(w_j\right)\right)\right|> 2$$ 
Then
$$\displaystyle\sum_{i\neq j}\left|\ell\left(link\left(v_i\right)\right)\cap\ell\left(\link\left(v_j\right)\right)\right|\ \geq\ \displaystyle\sum_{i\neq j}\left|\ell\left(\link\left(w_i\right)\right)\cap\ell\left(\link\left(w_j\right)\right)\right|> 2$$
which is a contradiction.
\end{proof}
\begin{lem}\label{lem:degree of w}
Let $\phi: H\rightarrow K$ be an immersion of core graphs and let $w\in K^0$ be such that $S=\phi^{-1}\left(w\right)$ is inert. Then
\begin{equation}\label{eq:degree of w}
\deg\left(w\right)-2\ \geq\ \displaystyle\sum_{v\in S}\left(\deg\left(v\right)-2\right).
\end{equation}
Consequently, 
\begin{equation}\label{eq:curvature of w}
\kappa\left(w\right)\leq \displaystyle\sum_{v\in S}\kappa\left(v\right).
\end{equation}

\end{lem}
\noindent See Figure~\ref{fig:graph_9} for an illustration.
\begin{figure}[t]\centering
\includegraphics[width=.7\textwidth]{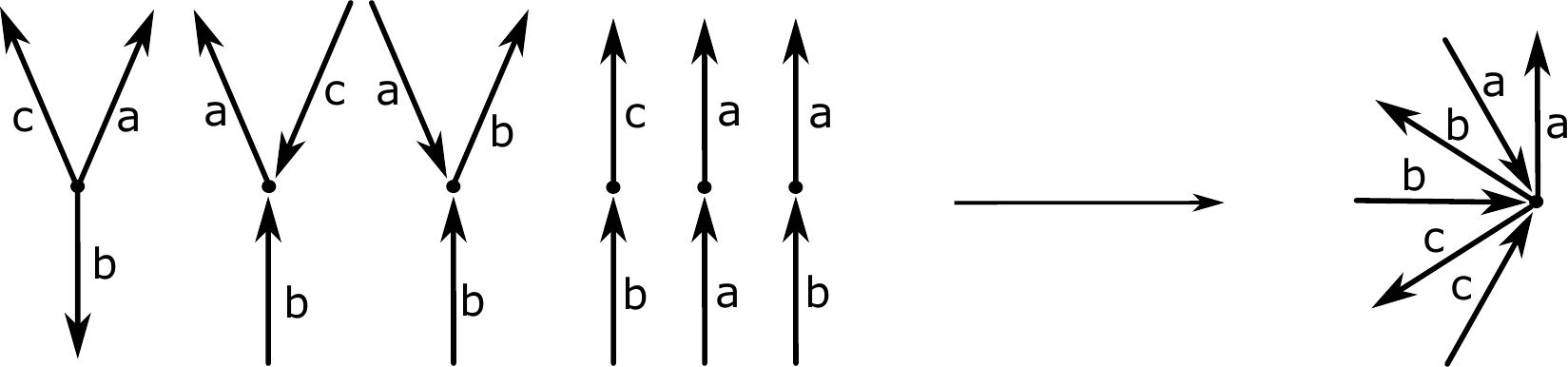}
\caption[Immersion of an inert set $S$ ]{\label{fig:graph_9}
Immersion of an inert set $S$.}
\end{figure}
\begin{proof}
If $S=\emptyset$, then both inequalities hold trivially since $K$ is a core graph. Otherwise, Inequality~\ref{eq:degree of w} is a consequence of the following argument:\\
Let $\left\{X_i\ :\ 0\leq i\leq m\right\}$ and $\left\{X'_i\ :\ X'_i\subset X_i,\ \left|X_i-X'_i\right|\leq 2,\ 0\leq i\leq m\right\}$ be collections of disjoint finite sets. Let $f:\displaystyle\bigsqcup_{i=1}^mX_i\rightarrow X$ be a function such that for all $0\leq i, j\leq m$ we have 
\begin{enumerate}
\item $f|_{X_i}$ is injective,
\item $f:\displaystyle\bigsqcup_{i=1}^mX'_i\rightarrow X$ is injective, and
\item $f\left(X_i-X_i'\right)\cap f\left(X_j'\right)=\emptyset$.
\end{enumerate}
 Then $$\left|X\right|\ \geq\ 2+\displaystyle\sum_{i=1}^m\left(\left|X_i\right|-2\right).$$
Inequality~\ref{eq:curvature of w} follows from Inequality~\ref{eq:degree of w} as shown below
\begin{align*}
\kappa\left(w\right)=\pi \left(2-\deg\left(w\right)\right)&\leq \pi \left(2-\displaystyle\sum_{v\in S} \left(\deg\left(v\right)-2\right)-2\right)\\
&=\pi \displaystyle\sum_{v\in S}\left(2-\deg\left(v\right)\right)\\
&=\displaystyle\sum_{v\in S}\kappa\left(v\right) \qedhere
\end{align*}
\end{proof}
\begin{prop}\label{prop:highly inert is inert}
Let $H\rightarrow B$ be a highly inert immersion. Then $H\rightarrow B$ is inert.
\end{prop}
\begin{proof}
Let $K\rightarrow B$ be an immersion with $K$ a compact core graph. Let $A$ be the based core of $H\otimes_B K$. Consider the commutative diagram:
\[
\begin{tikzcd}
A \arrow[r,"\alpha"] \arrow[d,black,"\beta"]
& K \arrow[d] \\
H \arrow[r]
& B
\end{tikzcd}
\]
Let $w\in \alpha\left(A^0\right)$. By Remark~\ref{rem:cardinality of fibers}, the map $\beta|_{\alpha^{-1}\left(w\right)}$ is injective. Since $\beta$ is an immersion, by Lemma~\ref{lem:composition of highly inert}, the set $\alpha^{-1}\left(w\right)$ is inert. By Lemma~\ref{lem:degree of w}, we have:  
$$\kappa\left(w\right)\leq \displaystyle\sum_{v\in \alpha^{-1}\left(w\right)}\kappa\left(v\right).$$
Then by Theorem~\ref{thm:GBT}, we have the following, where the first inequality is an equality when $\alpha$ is surjective, and the second inequality is by Lemma~\ref{lem:degree of w}: 
\begin{align*}
2\pi \chi\left(K\right)&=\displaystyle\sum_{w\in K^0}\kappa\left(w\right)\\
&\leq\displaystyle\sum_{w\in \left(K^0\cap \alpha\left(A^0\right)\right)}\kappa\left(w\right)\\
&\leq \displaystyle\sum_{w\in \left(K^0\cap \alpha\left(A^0\right)\right)}\left(\displaystyle\sum_{v\in \left(\alpha^{-1}\left(w\right)\cap A^0\right)}\kappa\left(v\right)\right)\\
&=\displaystyle\sum_{v\in A^0}\kappa\left(v\right)\\
&=2\pi \chi\left(A\right).
\end{align*}
On the other hand, both $A$ and $K$ are connected graphs, so

$$2\pi \chi\left(K\right)=2\pi \left(1-\rank\left(K\right)\right)\leq 2\pi \chi\left(A\right)=2\pi\left(1-\rank\left(A\right)\right).$$
Therefore, 
\begin{equation*}
\rank\left(A\right)\leq \rank\left(K\right).\qedhere
\end{equation*}
\end{proof} 
\begin{cor}\label{cor:one branching vertex}
Let $H\rightarrow B$ be the immersed graph. Then $$\left|H^*\right|\leq 1\ \Rightarrow\ H\ \text{is inert}.$$
In particular, cycles and immersed bouquets of circles are inert.
 \end{cor}
 \noindent This is easy to see since given any immersed graph $K\rightarrow B$, the fibers of the immersion $H\otimes_BK\rightarrow K$ contain at most one branching vertex and thus form inert sets.
\begin{rem}\label{rem:rank two HNC}
Any graph of rank $2$ must have either one branching vertex of degree $4$ or two branching vertices of degree $3$. By Corollary~\ref{cor:one branching vertex}, the former is inert and the latter is inert whenever its two branching vertices have distinct links. So the proposition provides a partial proof of the Hanna Neumann conjecture when one of the subgroups has rank $2$. \end{rem}


\subsection{Compressed Subgroups} 
\begin{defn}\label{defn:compressed subgroup}
A subgroup $\mathcal{H}\leq \mathcal{F}$ is \textit{compressed} in $\mathcal{F}$ if the following holds for all subgroups $\mathcal{K}\leq\mathcal{F}$:
$$\mathcal{H}\leq \mathcal{K}\  \Rightarrow\  \rank\left(\mathcal{H}\right)\leq \rank\left(\mathcal{K}\right).$$
Likewise, an immersed connected graph $H\rightarrow B$ is \textit{compressed} if for any immersed graph $K\rightarrow B$ we have  $$\left(H\rightarrow K\right)\  \Rightarrow\ \left(\rank\left(H\right)\leq \rank\left(K\right)\right)$$
where $H\rightarrow K$ is an immersion.
\end{defn}
\noindent Note that $H$ fails to be compressed if and only if there exists a surjective immersion $H\twoheadrightarrow K$ such that $\rank\left(H\right)>\rank\left(K\right)$.\\
One can see that inert subgroups are compressed. Indeed, suppose we have subgroups $\mathcal{H}\leq \mathcal{K}\leq \mathcal{F}$ with $\mathcal{H}$ inert, then if $\rank\left(\mathcal{H}\right)>\rank\left(\mathcal{K}\right)$ then $\mathcal{H}\cap\mathcal{K}=\mathcal{H}$ and thus $\rank\left(\mathcal{H}\cap \mathcal{K}\right)>\rank\left(\mathcal{K}\right)$, contradicting the inertness of $H$.
\begin{lem}\label{lem:compressed iff}
Let $H\rightarrow B$ be an immersion representing the subgroup $\mathcal{H}$. Then
$$H\rightarrow B\ \text{compressed}\ \iff\ \mathcal{H}\ \text{compressed}$$
\end{lem}
\noindent The proof is similar to the proof of Lemma~\ref{lem:inert iff}.
The following lemma is taken from \cite{MR2097438}.
\begin{lem}\label{lem:free factor}
Let $\mathcal{H}\leq \mathcal{F}$ and suppose $\mathcal{H}=\mathcal{A}\ast \mathcal{B}$. If $\mathcal{H}$ is compressed then $\mathcal{A}$ is compressed.
\end{lem}
\noindent For graphs, Lemma~\ref{lem:free factor} is analogous to the statement that connected subgraphs of a compressed graph are compressed. See Lemma~\ref{lem:compressed subgraph} for a strong form of this statement for graphs.
\begin{defn}\label{defn:arc}
An \textit{arc} $A\subset H$ is a connected component of $H-H^*$. The \textit{boundary vertices} of $A$, denoted by $\partial A_1,\ \partial A_2$, are the branching vertices in the closure of $A$.
\end{defn}
\noindent Let $\ell:H\rightarrow B$ be an immersion, and $e$ be an edge in an arc $A$ such that $$\left\{\ell\left(e\right)\right\}\cap \ell\left(H-A\right)=\emptyset.$$ 
This means that the label of $e$ does not appear in $H-A$, although it may appear more than once in $A$. We can then construct a new graph $\bar{H}$ by 
replacing $A$ by $e$:
 $$\bar{H}=\left(H-A\right)\sqcup \left\{e\right\},$$
 where we attach $o\left(e\right)$ to $\partial A_1$, and $\tau\left(e\right)$ to $\partial A_2$. Note that $\bar{H}\rightarrow B$ is an immersion since $\ell\left(e\right)$ does not appear elsewhere. Such an edge $e$ is a \textit{placeholder}. The graph $\bar{H}$ is called a $\left(A\searrow e\right)$-\textit{deflation} of $H$. Likewise, the graph $H$ obtained from $\bar{H}$ by replacing an edge $e$ with an arc $A$ is a $\left(e\nearrow A\right)$-\textit{inflation} of $\bar{H}$. See Figure~\ref{fig:graph_7}. In this setting, we have the following lemma:
 \begin{figure}[t]\centering
\includegraphics[width=.8\textwidth]{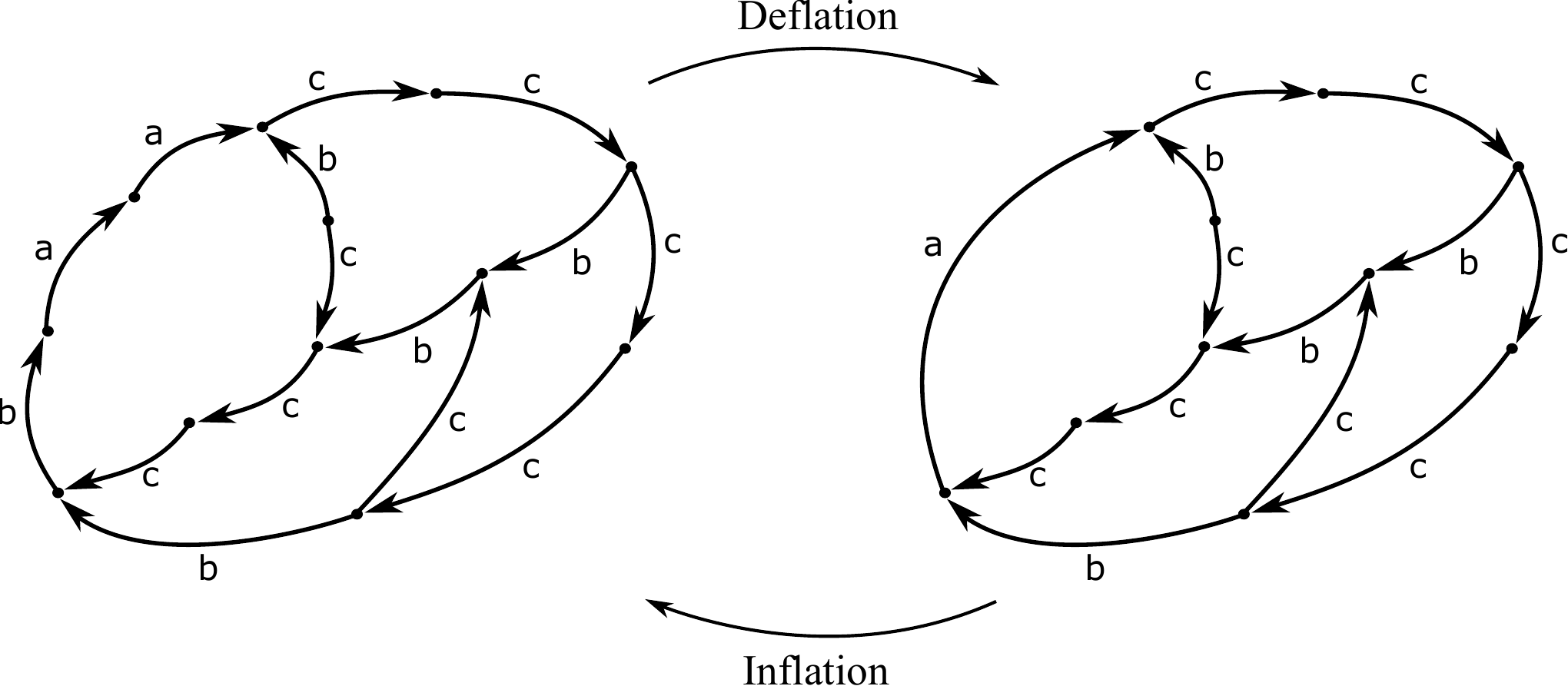}
\caption[Inflation-Deflation of graphs ]{\label{fig:graph_7}
Inflation-Deflation of a graph.}
\end{figure}
\begin{lem}[Placeholder]\label{lem:placeholder}
Let $\ell:H\rightarrow B$ be an immersion, $e$ be an edge in an arc $A\subset H$ such that $\left\{\ell\left(e\right)\right\}\cap \ell\left(H-A\right)=\emptyset$.
Let $\bar{H}$ be the $\left(A\searrow e\right)$-\textit{deflation} of $H$.
Then $\rank\left(\bar{H}\right)=\rank\left(H\right)$, and if $H$ is compressed so is $\bar{H}$.
\end{lem}
\begin{proof}
Since there is a degree preserving bijection $\bar{H}^*\rightarrow H^*$, we have $\chi\left(\bar{H}\right)=\chi\left(H\right)$ by Theorem~\ref{thm:GBT}. Therefore, $\rank\left(\bar{H}\right)=\rank\left(H\right)$.\\
Suppose $\bar{H}$ is not compressed. Then there exists an immersion $\bar{g}: \bar{H} \rightarrow \bar{K}$ with $\rank\left(\bar{H}\right)>\rank\left(\bar{K}\right)$. Let $\bar{K}'$ be the $\left(\bar{g}\left(e\right)\nearrow A\right)$-\textit{inflation} of $\bar{K}$, and $\bar{K}'\xrightarrow{F}K$ be the composition of all possible foldings. By Lemma~\ref{lem:folding}, foldings are $\pi_1$-surjective, so we have
$$\rank\left(H\right)=\rank\left(\bar{H}\right)>\rank\left(\bar{K}\right)=\rank\left(\bar{K}'\right)\geq \rank\left(K\right)$$
One checks that the map $g: H\rightarrow K$ defined as $g|_{H-A}=F\circ \bar{g}$ and $g|_A=F\circ \id_A$ is an immersion, contradicting the compressibility of $H$. 
\end{proof}
\begin{lem}[Generalized Placeholder]\label{cor:generalized}
Let $H$ be a compressed graph of rank $>1$. Let $e_1,\ldots,e_m \in H^1$ be edges in the arcs $A_1\ldots,A_m$ respectively, such that
\begin{enumerate}
\item $\ell\left(e_i\right)\neq \ell\left(e_j\right)$ whenever $i\neq j$.
\item $\left\{\ell\left(e_i\right)\right\}\cap \ell\left(H-\left\{A_1,\ldots,A_m\right\}\right)=\emptyset$ for all $1\leq i\leq m$.
\end{enumerate}
Let $\bar{H}$ be the $\left(A_i\searrow e_i\right)_{i=1}^{m}$-deflation of $H$, that is $\bar{H}$ is the graph obtained by replacing arcs $A_i$ with edges $e_i$. Then $\bar{H}\rightarrow B$ is an immersion, $\bar{H}$ is compressed, and $\rank\left(\bar{H}\right)=\rank\left(H\right)$.
\end{lem} 
\noindent The set $\left\{e_1,\ldots,e_n\right\}$ is called a \textit{placeholder set}.
The proof is similar to that of Lemma~\ref{lem:placeholder}. \\ 

We now describe a combinatorial lemma. Let $\mathcal{M}\twoheadrightarrow \mathcal{U}$ be a function describing the storage of a finite set of marbles $\mathcal{M}$ in $k$ urns, $\mathcal{U}=\left\{U_1,\ldots,U_k\right\}$, such that no urn is empty. Let $\mathcal{C}=\left\{C_1,\ldots,C_m\right\}$ be a set of $m$ distinct colours with $m\leq k$, and $c:\mathcal{M}\twoheadrightarrow \mathcal{C}$ be the colouring map. A \textit{transversal} $\mathcal{T}$ is the image of a section of $\mathcal{M}\rightarrow \mathcal{U}$ where the composition $\mathcal{U}\rightarrow\mathcal{M}\rightarrow \mathcal{U}$ is $\id_\mathcal{U}$. In other words, $\mathcal{T}$ is a set of marbles constructed by choosing exactly one marble from each urn. Each chosen marble is a \textit{representative} of the urn from which it was chosen. So $\left|\mathcal{T}\right|=k$. We are interested in the cardinality of $c\left(\mathcal{T}\right)$. Suppose that all transversals are such that $\left|c\left(\mathcal{T}\right)\right|<m$. Then by finiteness, we can choose $\mathcal{T}$ with $\left|c\left(\mathcal{T}\right)\right|=s<m$ where $s$ is the maximal number of colours that can possibly appear in a transversal. Let $\mathcal{T}=\left\{M_i\mid M_i\ \text{is a representative of the urn}\ U_i,\ 1\leq i\leq k\right\}$. There are $m-s>0$ colours missing from $\mathcal{T}$. We let $X$ be the set of these missing colours and call its elements \textit{$x$-colours}. The $x$-coloured marbles must appear in at least one urn, since $c:\mathcal{M}\twoheadrightarrow \mathcal{C}$ is surjective. We let $Y$ be the set of the colours of representatives of urns containing $x$-coloured marbles. Elements of $Y$ are \textit{$y$-colours}. Note that $\left|X\right|\geq 1$ and $\left|Y\right|\geq 1$. See Figure~\ref{fig:graph_14}. In this setting we have the following lemma:
\begin{lem}\label{lem:urn}
There exists a subset $\mathcal{T}'\subset\mathcal{T}$ such that
\begin{enumerate}
\item For any marble $M\in \mathcal{T}',\ c\left(M\right)\notin c\left(U_i\right)$ where $U_i$ is any urn represented by $M_i\in \left(\mathcal{T}-\mathcal{T}'\right)$, and
\item $1\leq \left|\mathcal{T}'\right|\leq \left|\mathcal{T}\right|-2$.
\end{enumerate}
\end{lem}
\begin{figure}[t]\centering
\includegraphics[width=.4\textwidth]{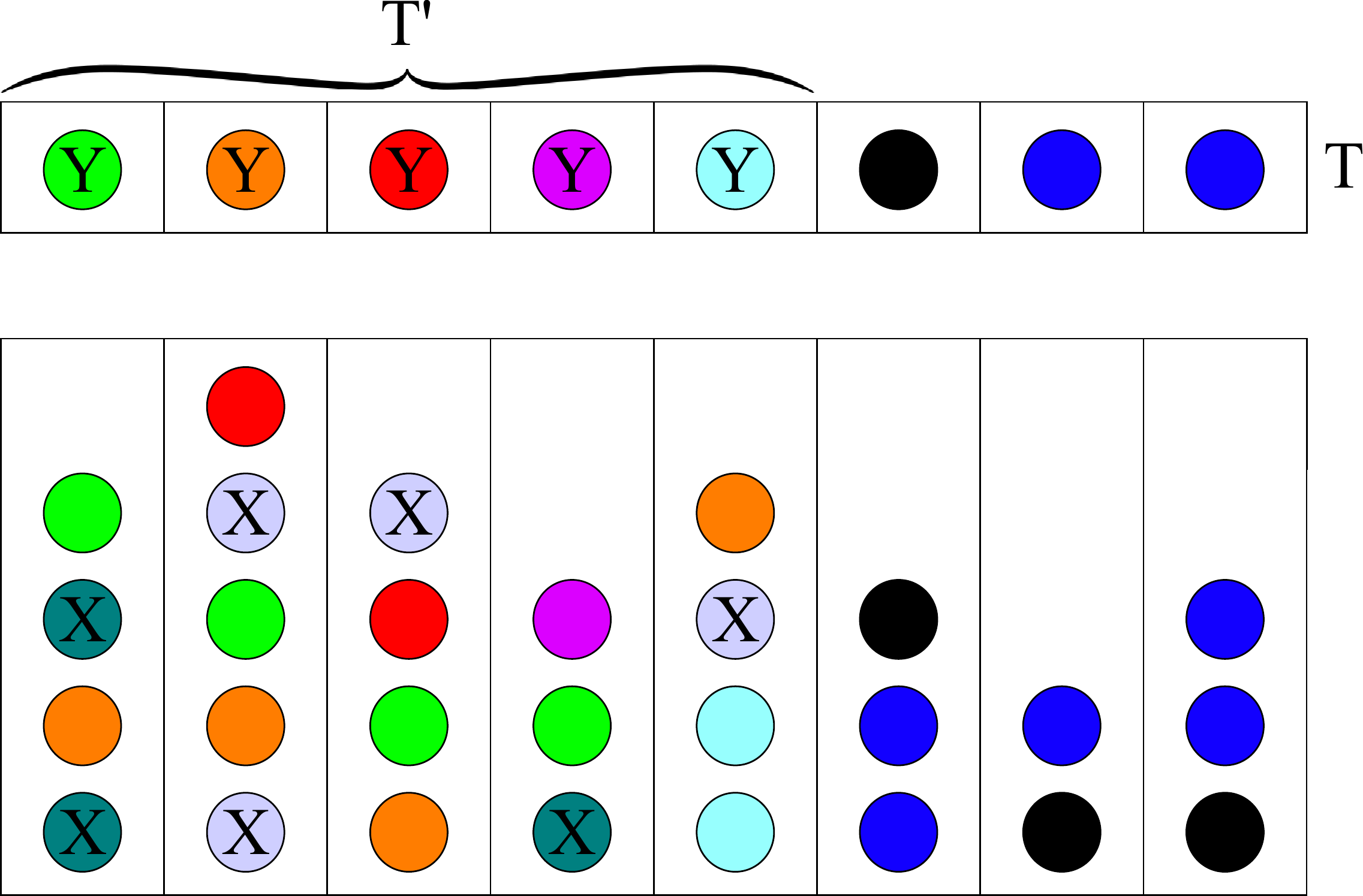}
\caption[Urn Lemma]{\label{fig:graph_14}
The urns and their transversal.}
\end{figure}
\begin{proof}
Since we have $k-s\geq m-s\geq 1$ there is a least one urn $U_j$ whose representative $M_j$ has a colour that appears elsewhere in $\mathcal{T}$, ie $c\left(M_j\right)\in c\left(\mathcal{T}-\left\{M_j\right\}\right).$ Let $U_{j'}$ be the urn such that $c\left(M_j\right)=c\left(M_{j'}\right)$, with $M_j, M_{j'}\in \mathcal{T}$. Then $U_j$ and $U_{j'}$ contain no $x$-coloured marbles and no $y$-coloured marbles. Indeed, if there exists $M\in U_j$ such that $c\left(M\right)\in X$, then we can replace $M_j$ by $M$ as the representative of $U_j$ and obtain a new transversal that contains $s+1$ colours, contradicting the maximality assumption. Similarly, if there exists $M\in U_j$ such that $c\left(M\right)\in Y$, then letting $M_i$ be the $y$-coloured representative such that $c\left(M\right)=c\left(M_i\right)$, we can construct a new transversal containing $s+1$ colours by replacing $M_j$ with $M$ as the representative of $U_j$, and by replacing $M_i$ by an $x$-coloured marble in $U_i$. This again contradicts the maximality assumption. The same is true for $U_{j'}$.\\
Let $\mathcal{T}'=\left\{M_i\in \mathcal{T}\mid \exists\ M\in U_i\ \text{with}\  c\left(M\right)\in Y\right\}$. Then $\mathcal{T}'\neq \emptyset$ and $U_j,\ U_{j'}\notin \mathcal{T}'$. Hence $\mathcal{T}'$ is the desired set.
\end{proof}
\noindent Let $H\twoheadrightarrow B$ be a compressed immersed core graph of rank $m>1$, and let $\mathcal{A}=\left\{A_1,\ldots,A_k\right\}$ be its collection of arcs. Note that $m\leq k$, with equality holding if and only if $H$ is an immersed (subdivided) bouquet of circles. Each edge belongs to a unique arc. We define an equivalence relation $\sim$ on the edge set $H^1$ by $$e\sim e' \iff \exists A\subset \mathcal{A}\ \text{such that}\ \left\{e,e'\right\}\subset A$$ 
Let $\mathcal{T}=\left\{e_1,\ldots,e_k\mid e_i\in A_i\right\}$ be a transversal of this equivalence relation, where each $e_i$ is an edge representing the arc  $A_i$. Then we have the following proposition:


\begin{prop}\label{prop:label distribution}
If $\ell:H\rightarrow B$ is an immersed compressed graph of rank $m$, then it admits a transversal $\mathcal{T}$ such that $\left|\ell\left(\mathcal{T}\right)\right|\geq m$.
\end{prop}
\begin{proof}
We first remark that since $H$ is compressed and has rank $m$, then we necessarily have $m\leq n$, because otherwise $H$ is immersed in a smaller rank bouquet of circles which is a contradiction. Moreover, $H$ must contain at least $m$ edges each one having a different label, for otherwise once again, $H$ would immerse in a smaller rank bouquet of circles.\\
We need to show that there exists a transversal $\mathcal{T}$ such that $\left|\ell\left(\mathcal{T}\right)\right|\geq m$. Since we have only finitely many possible transversals, we can choose one whose image under $\ell$ is maximal. Let it be $\mathcal{T}$ and let $s=\left|\ell\left(\mathcal{T}\right)\right|$. So  for any transversal $\mathcal{T}'$ we have $\left|\ell\left(\mathcal{T}'\right)\right|\leq s$.\\ 
Suppose $s<m$. By Lemma~\ref{lem:urn}, there exists a subset $\mathcal{T}'\subset\mathcal{T}$ such that for any edge $e\in \mathcal{T}',\ \ell\left(e\right)\notin \ell\left(A_i\right)$ for all $A_i$ represented by $e_i\in \left(\mathcal{T}-\mathcal{T}'\right)$. Then $\mathcal{T}'$ is a placeholder set. Let $\bar{H}$ be the corresponding $\left(A_j\searrow e_j\right)_{e_j\in \mathcal{T}'}$ deflation of $H$ as constructed in Lemma~\ref{cor:generalized}. Then $\bar{H}$ is compressed. However, this is a contradiction since $\left|\ell\left(\bar{H}^1\right)\right|=s<m$, which means $\bar{H}$ immerses in a smaller rank bouquet of circles.
\end{proof}
\noindent Proposition~\ref{prop:label distribution} shows that compressed labelled graphs have constraints on how their labelled edges are distributed in the graph. In particular, it shows that in some sense, \emph{too many} labels cannot be concentrated in \textit{too few} arcs. We will say more about the distributions of labelled edges of compressed graphs in the next section.
\subsection{Compressed Disconnected Graphs}
\noindent We now generalize the notion of compressibility to disconnected finite graphs.
\begin{defn}
A finite graph $H$ is \textit{compressed} if for any immersion of graphs $H\rightarrow K$ we have $\mrank\left(H\right)\leq\mrank\left(K\right)$.
\end{defn}
\noindent When $H$ is connected this definition coincides with Definition~\ref{defn:compressed subgroup}.
\begin{lem}
Let $H=\displaystyle\bigsqcup_{i=1}^nH_i$ where each $H_i$ is connected. If $H$ is compressed then $H_i$ is compressed for all $1\leq i\leq n$.
\end{lem}
\begin{proof}
Suppose $H_j$ is not compressed for some index $j$. Then there exists an immersion of connected graphs 
$\phi:H_j\rightarrow K$ such that $\rank\left(H_j\right)> \rank\left(K\right)$. Consider the graph $\Gamma=K\sqcup \displaystyle\bigsqcup_{i\neq j}H_i$ and the map $\Phi: H\rightarrow \Gamma$ defined as follows: $\Phi_{|H_{j}}=\phi$ and $\Phi_{|H_{i\neq j}}=\id_{H_{i\neq j}}$. One checks that $\Phi$ is an immersion and $\mrank\left(H\right)> \mrank\left(\Gamma\right)$ contradicting the compressibility of $H$.
\end{proof}
\noindent The converse of this lemma is not true. For example if $H_1$ is compressed of rank $m\geq 2$, then the graph $H=H_1\sqcup H_1$ (a disjoint union of two copies of $H_1$) is not compressed since $H$ immerses onto $H_1$. However, as Lemma~\ref{lem:compressed subgraph} shows, the disjoint union of graphs $H_i$ which are all subgraphs of a compressed graph $H$, is compressed. In particular, it shows that connected subgraphs of a compressed graph are compressed, thus providing another proof to Lemma~\ref{lem:free factor}

\begin{lem}\label{lem:compressed subgraph}
Let $H$ be a compressed connected graph, and $H'=\displaystyle\bigsqcup_{i=1}^nH_i$ be a union of disjoint connected subgraphs $H_i\subset H$. Then $H'$ is compressed. In particular, connected subgraphs of $H$ are compressed.
\end{lem}
\begin{proof}
Suppose $\phi: H'\rightarrow K$ is an immersion such that$$\mrank\left(H'\right)=\displaystyle\sum_{i=1}^n\mrank\left(H_i\right)>\mrank\left(K\right).$$ Then $\chi\left(H'\right)<\chi\left(K\right)$. Let $\overline{\left(H-H'\right)}$ be the closure of $\left(H-H'\right)$ in $H$. Denote by $\mathcal{V}\subset H^0$ the set $\mathcal{V}=\overline{\left(H-H'\right)}\cap\left(H'\right)$.
So $\mathcal{V}$ is precisely the set of vertices at the boundaries of both $\overline{\left(H-H'\right)}$ and $H'$. Consider the graph $$\Gamma'=\frac{\left(H-H'\right)\sqcup K}{v\sim \phi\left(v\right)}\hspace{5mm}\forall\ v\in \mathcal{V}$$
where $v\sim \phi\left(v\right)$ is the identification of $v$ and $\phi\left(v\right)$. Let $F:\Gamma'\rightarrow \Gamma$ be the folding of $\Gamma'$ into $\Gamma$. Then
$$\chi\left(H\right)=\chi\left(H-H'\right)+\chi\left(H'\right)<\chi\left(H-H'\right)+\chi\left(K\right)=\chi\left(\Gamma'\right)\leq \chi\left(\Gamma\right)$$
Let $\Phi: H\rightarrow \Gamma$ be a map defined by $\Phi|_{H-H'}=F\circ\id|_{H-H'}$ and $\Phi|_{H'}=F\circ\phi$. By construction $\Phi$ maps vertices to vertices and edges to edges. Moreover, one checks that it is an immersion. Thus we have $\rank\left(H\right)>\rank\left(\Gamma\right)$, which contradicts the assumption.
\end{proof}
\begin{defn}
 Let $H$ be a finite graph. An edge $e\in H^1$ is \textit{essential} if $$\mrank\left(H-\left\{e\right\}\right)=\mrank\left(H\right)-1.$$
 A subset $\mathcal{E} \subset H^1$ is essential if $$\mrank\left(H-\mathcal{E}\right)=\mrank\left(H\right)-|\mathcal{E}|.$$
 $\mathcal{E}$ is \textit{maximal essential} if and only if $\mathcal{E}$ is essential and $\mrank\left(H\right)=|\mathcal{E}|$. When $\mathcal{E}$ is a maximal essential set, each connected component of $H-\mathcal{E}$ is an \textit{island}, and the set of all islands is denoted by $\mathcal{I}$. For an essential edge $e$, we denote by $C$ the core of the connected component of $\left(H-\mathcal{E}\right)\cup\left\{e\right\}$ containing $e$, and by $\mathcal{C}$ the set of such subgraphs. A graph may have several distinct maximal essential sets.\\
\noindent The following lemma was stated as a remark without proof in \cite{MR2914871}. We restate it and prove it below: 
\end{defn}
\begin{lem}\label{lem:essential set}
Let $H$ be a connected graph and $\mathcal{E}\subset H^1$ be a subset of edges. Then $\mathcal{E}$ is essential in $H$ if and only if  $\ \forall\  e \in \mathcal{E}$, the component of $\left(H-\mathcal{E}\right)$ containing $o\left(e\right)$ is not a tree and the component $\left(H-\mathcal{E}\right)$ containing $\tau\left(e\right)$ is not a tree. In particular, islands have rank $1$.
\end{lem}
\begin{proof}
Let $e\in \mathcal{E}$.\\
$\left(\Rightarrow\right)$ Clearly if either component is a tree then the edge $e$ is no longer essential, as removing or adding it does not decrease or increase the reduced rank by $1$, contradicting the assumption.\\
$\left(\Leftarrow\right)$ Let $H_1, H_2$ be the connected components of $\left(H-\mathcal{E}\right)$ containing $o\left(e\right)$ and $\tau\left(e\right)$.\\
If $H_1=H_2$, then there is a path $p:I_d\rightarrow H_1$ such that $p\left(d\right)=o\left(e\right)$ and $p\left(0\right)=\tau\left(e\right)$. By assumption, $\rank\left(H_1\right)\geq 1$. Then adding the edge $e$ to $H_1$ will yield a new non-trivial cycle which is the concatenation of $p$ and $e$. This would increase the rank by $1$, thereby increasing the reduced rank by $1$, making $e$ essential.\\
If $H_1\neq H_2$, then $\rank\left(H_1\right)\geq 1$ and $\rank\left(H_2\right)\geq 1$. Adding $e$ to connect $H_1$ and $H_2$ will reduce the number of components, thus increasing the reduced rank. So $e$ is essential.\\
When $\mathcal{E}$ is maximal essential, we have $\mrank\left(H-\mathcal{E}\right)=\displaystyle\sum_{I\in \mathcal{I}}\mrank\left(I\right)=0$. So $\mrank\left(I\right)=0$ for any island $I\in \mathcal{I}$. Since no component of $H-\mathcal{E}$ is a tree, islands must have rank $1$.  
\end{proof}
\noindent Let $\mathcal{E}$ be a maximal essential set of $H$, and $\mathcal{E}_1\subsetneq \mathcal{E}$ be a proper non-empty subset. Let $H_1$ be the (possibly disjoint) union of the connected components of $\left(H-\mathcal{E}\right)\cup \mathcal{E}_1$ containing $\mathcal{E}_1$. Let $\mathcal{E}_1'$ be any maximal essential set of $H_1$, and $\tilde{\mathcal{E}}=\left(\mathcal{E}-\mathcal{E}_1\right)\sqcup \mathcal{E}_1'$. See Figure~\ref{fig:graph_13}. 
\begin{figure}[t]\centering
\includegraphics[width=0.6\textwidth]{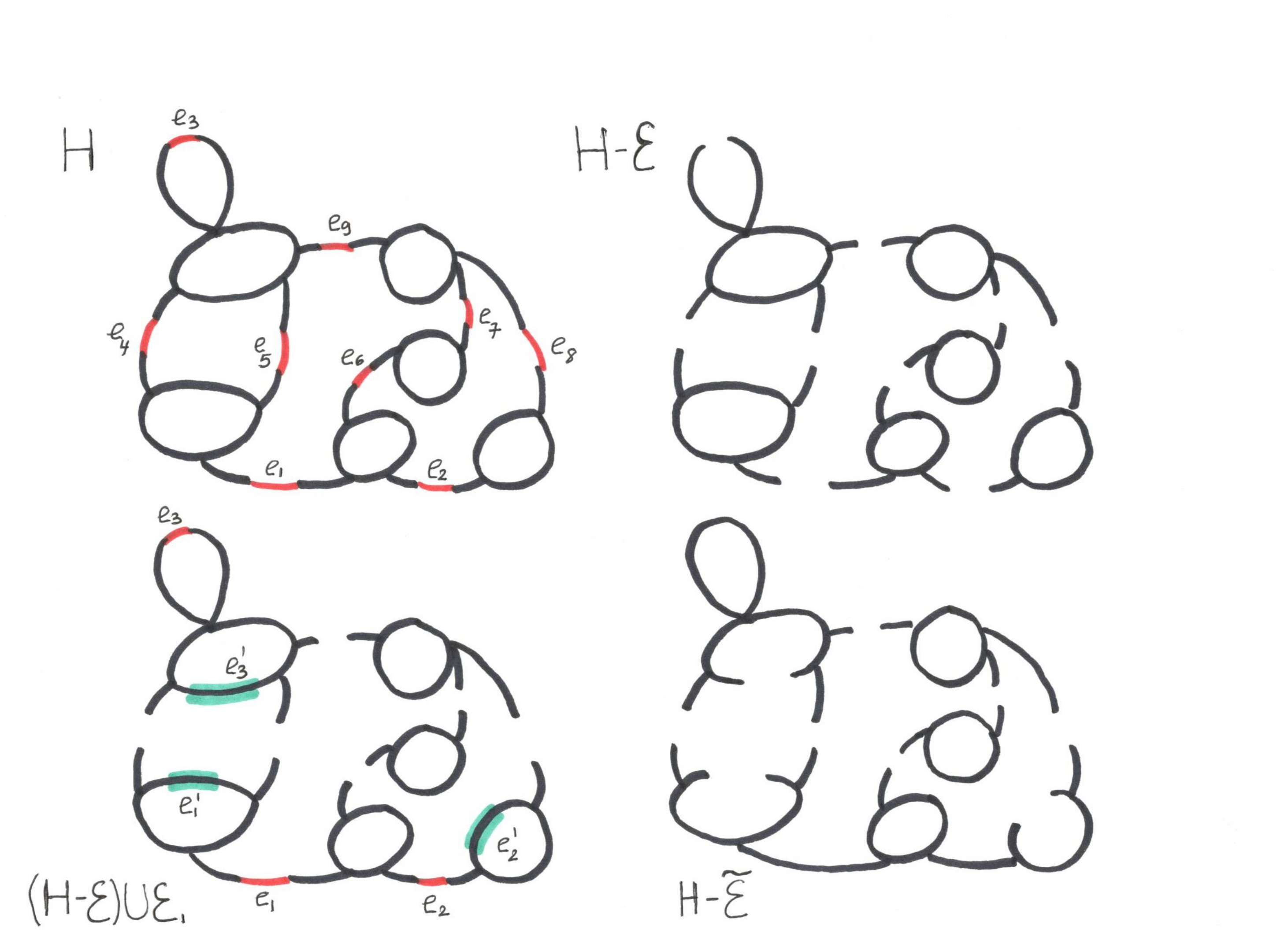}
\caption[Invariance of maximal essential sets]{\label{fig:graph_13}
In this example, we have $\mathcal{E}=\left\{e_1,\ldots,e_9\right\},\ \mathcal{E}_1=\left\{e_1,e_2,e_3\right\}$, and $\mathcal{E}_1'=\left\{e_1',e_2',e_3'\right\}$. Observe that $\mathcal{E}_1'$ is maximal essential in $H_1$, and $\mathcal{E}_1'\cup \left(\mathcal{E}-\mathcal{E}_1\right)=\tilde{\mathcal{E}}$ is maximal essential in $H$.}
\end{figure}
In this setting we have the following lemma:
\begin{lem}\label{lem:maximal essential}
$\tilde{\mathcal{E}}$ is maximal essential in $H$.
\end{lem}
\begin{proof}
We first note that $|\mathcal{E}_1|=\mrank\left(H_1\right)=|\mathcal{E}'_1|$. So $|\tilde{\mathcal{E}}|=\mrank\left(H\right)$. Let $e\in \tilde{\mathcal{E}}$. If $e\in \mathcal{E}_1'$, then by Lemma~\ref{lem:essential set}, $e$ is essential since the connected components containing its boundary vertices are islands in $H_1-\mathcal{E}_1'$. Otherwise, let $I, I'$ be the islands in $H-\mathcal{E}$ containing $o\left(e\right)$ and $\tau\left(e\right)$, respectively. (Note that it is possible to have $I=I'$). If $I, I'\not\subset H_1$, then $I, I'$ remain islands in $H-\tilde{\mathcal{E}}$, making $e$ essential.\\
Now suppose $I\subset H_1$, and $I_1$ is the connected component of $H_1-\mathcal{E}_1'$ containing $o\left(e\right)$. Then $I_1$ has necessarily rank $1$ since it is an island in $H-\tilde{\mathcal{E}}$, thus $e$ is essential.
\end{proof}
\noindent Let $\ell:H\rightarrow B$ be an immersion where $H$ is a finite compressed graph of rank $m$. Let $E=\left\{\mathcal{E}_1,\ldots,\mathcal{E}_k\right\}$ be the collection of all maximal essential sets of $H$, where $\mathcal{E}_i=\left\{e_{i,1}\ldots,e_{i,\left(m-1\right)}\right\}$ for all $1\leq i\leq k$.  To each $\mathcal{E}_i$ corresponds a set of connected core components $\mathcal{C}_i=\left\{C_{i,1},\ldots,C_{i,\left(m-1\right)}\right\}$, where $C_{i,j}$ is the core component of $\left(H-\mathcal{E}_i\right)\cup\left\{e_{i,j}\right\}$ containing $e_{i,j}$ with $\mrank\left(C_{i,j}\right)=1$ for all $1\leq i\leq k$ and $1\leq j\leq m-1$. Then in this setting we have the following theorem:
\begin{thm}\label{thm:essential distribution}
There exists a maximal essential set $\mathcal{E}$ such that $\ell|_{\mathcal{E}}$ is injective.
\end{thm}
\begin{proof}
Let $\mathcal{E}\in E$ be such that $$\left|\ell\left(\mathcal{E}\right)\right|=s=\max\left\{\left|\ell\left(\mathcal{E}_i\right)\right|,\ 1\leq i\leq k\right\}.$$
Suppose $s<m-1$. Then there exist at least two edges $e_1, e_2\in \mathcal{E}$ such that $\ell\left(e_1\right)=\ell\left(e_2\right)$. Let $C_1, C_2$ be their corresponding core components, respectively. 
Since $H\rightarrow B$ is compressed, $\left|\ell\left(H^1\right)\right|\geq m$.
As in Lemma~\ref{lem:urn}, let $X=B^1-\ell\left(\mathcal{E}\right)$ and  $Y= \left\{\ell\left(e_j\right) \in B^1\mid e_j\in \mathcal{E},\ \text{and}\ \exists e\in C_j, \ell\left(e\right)\in X\right\}.$ 
Then $\ell\left(C_1\right)\cap X=\emptyset$, because if there exists an edge $e\in C_1$ such that $\ell\left(e\right)\in X$, then 
\begin{enumerate}
\item $e$ is essential in $C_1$ since $C_1$ is a core subgraph of reduced rank $1$ and removing any edge decreases the reduced rank to $0$.
\item By Lemma~\ref{lem:maximal essential}, $\tilde{\mathcal{E}}=\left(\mathcal{E}-\left\{e_1\right\}\right)\cup\left\{e\right\}$ is maximal essential. 
\end{enumerate}
Thus we have $\left|\ell\left(\tilde{\mathcal{E}}\right)\right|=s+1$, contradicting the maximality of $\left|\ell\left(\mathcal{E}\right)\right|$. Similarly, $\ell\left(C_1\right)\cap Y=\emptyset$. To see this, suppose that there exists an edge $e\in C_1$ such that $\ell\left(e\right)\in Y$. Let $e_j\in \mathcal{E}$ be such that $\ell\left(e\right)=\ell\left(e_j\right)\in Y$. Then there exists an edge $e_x\in C_j$ such that $\ell\left(e_x\right)\in X$. Since $C_1$ and $C_j$ are core subgraphs, both $e\in C_1$ and $e_x\in C_j$ are essential. In particular, $\mathcal{E}'=\left\{e,e_x\right\}$ is a maximal essential set of $C_1\cup C_j$. So $\tilde{\mathcal{E}}=\left(\mathcal{E}-\left\{e_1,e_j\right\}\right)\cup \left\{e,e_x\right\}$ is a maximal essential set of $H$ with $\left|\ell\left(\tilde{\mathcal{E}}\right)\right|=s+1$, which is a contradiction.\\
Now let $\mathcal{E}=\mathcal{E}_1\cup\mathcal{E}_2$ where $$\mathcal{E}_1=\left\{e_j\in \mathcal{E}\mid \ell\left(C_j\right)\cap Y\neq \emptyset\right\}$$ and  $$\mathcal{E}_2=\left\{e_i\in \mathcal{E}\mid \ell\left(C_i\right)\cap Y= \emptyset\right\}.$$
Then $$\mrank\left(\displaystyle\bigcup_{e_i\in \mathcal{E}_2}C_i\right)=m-1-\left|\mathcal{E}_1\right|.$$

\vspace{5mm} 
\noindent However, since $\left|\ell\left(\mathcal{E}\right)\right|=s<m-1$, we have 
\vspace{5mm} 
\begin{equation}\label{eq:contradiction}
\left|\ell\left(\displaystyle\bigcup_{e_i\in \mathcal{E}_2}C_i\right)\right|=s-\left|\mathcal{E}_1\right|<m-1-\left|\mathcal{E}_1\right|=\mrank\left(\displaystyle\bigcup_{e_i\in \mathcal{E}_2}C_i\right).
\end{equation}
But this is a contradiction since the union $\displaystyle\bigcup_{e_i\in \mathcal{E}_2}C_i$ is compressed by Lemma~\ref{lem:compressed subgraph}, and Inequality~\ref{eq:contradiction} says that this union surjects onto a bouquet of smaller reduced rank.
\end{proof}

%

\section{Echelon and Generalized Echelon Subgroups}\label{sec:ech}
\subsection{Orderability and the Hanna Neumann Conjecture}\label{sec:D}
\noindent The Hanna Neumann Conjecture (HNC) \cite{MR93537} states that for any finitely generated subgroups $\mathcal{H}$ and $\mathcal{K}$ of a free group $\mathcal{F}$, the following holds: $$\mrank\left(\mathcal{H}\cap\mathcal{K}\right)\leq \mrank\left(\mathcal{H}\right)\cdot\mrank\left(\mathcal{K}\right).$$ We will use some of the machinery introduced in the Mineyev-Dicks proof of HNC \cite{MR2914871} to 
generate a class of inert graphs which we call \textit{generalized echelon} graphs, and show that they are a generalization of \textit{echelon graphs} whose fundamental groups were defined by Rosenmann in \cite{MR3245107}. We begin by giving a summary of some definitions and facts taken from the proof of HNC, while the reader is referred to \cite{MR2914871} for more details.
\begin{defn}
A \textit{total ordering} of a set $X$, is a transitive binary relation $<$ such that for any $x, y\in X$, exactly one of the following holds:  $x<y$, $y<x$, or $x=y$. We refer to $X$ as an \textit{ordered set}. 
\end{defn}
\begin{rem}\label{rem:lexico}
Let $<_{_X},\ <_{_Y}$ be total orders on sets $X$ and $Y$ respectively. A \textit{lexicographic} order on the product $X\times Y$ is the total order $<$ defined as $$\left[\left(x_1,y_1\right)<\left(x_2,y_2\right)\right]\ \iff\ \left[\left(x_1<_{_X}x_2\right)\  \text{or}\ \left(y_1<_{_Y}y_2\ \text{if}\ x_1=x_2\right)\right].$$
\end{rem}
\begin{defn}
A group $G$ is \textit{left-orderable} if there exists a total order $<$ on $G$ that is invariant under left multiplication; ie: $\left(g_1<g_2\right)\ \Rightarrow \left(gg_1<gg_2\right)$ for all $g,g_1,g_2\in G$.  
\end{defn}
\begin{thm}
Free groups are left-orderable. 
\end{thm}
\noindent In fact, the space of left-orders of free groups is uncountable. For a proof and more about orderablity, we refer the reader to \cite{https://doi.org/10.48550/arxiv.1408.5805}.\\
A graph $\Gamma$ is an \textit{ordered graph} if there is a total order on $\Gamma^1$. Let $\mathcal{G}$ be a group acting freely and cocompactly on an ordered graph $\Gamma$ such that the order of $\Gamma^1$ is invariant under this action. Then $\Gamma$ is a \textit{free cocompact ordered} $\mathcal{G}$-\textit{graph}. Henceforth, we will refer to such graphs simply as ordered $\mathcal{G}$-graphs. If $\Gamma$ is a tree (forest), then it is an ordered $\mathcal{G}$-tree (forest).\\
If $\widetilde{B}$ is the universal cover of $B$, then $\mathcal{F}$ acts freely and cocompactly on it by left multiplication. Let $\mathcal{D}$ be the set of representatives of each orbit of the action of $\mathcal{F}$ on $\widetilde{B}^1$. Note that $\mathcal{D}$ is in one-to-one correspondence with $B^1$. Let $<_\mathcal{_D}$ be any total order on $\mathcal{D}$ and $<_\mathcal{_F}$ be any total order on $\mathcal{F}$. As in Definition~\ref{rem:lexico}, let $<$ be the lexicographic order on the edge set $\widetilde{B}^1\cong\mathcal{F}\times \mathcal{D}$. Then this order is invariant under the action of $\mathcal{F}$, and thus $\widetilde{B}$ is an ordered $\mathcal{F}$-tree.\\
Let $\widetilde{B}$ be an ordered $\mathcal{F}$-tree. A \textit{line} $L\subset \widetilde{B}$ is a subtree in which each vertex has degree $2$. An edge $e\in \widetilde{B}^1$ is a \textit{bridge} if it is the largest edge in some line $L$. Let $\mathcal{B}\widetilde{B}$ be the set of bridges of $\widetilde{B}$. Then $\mathcal{F}$ acts on $\mathcal{B}\widetilde{B}$ and $\left|\mathcal{F}\setminus \mathcal{B}\widetilde{B}\right|<\infty$.\\ 
Let $\mathcal{H}\leq \mathcal{F}$ be a non trivial finitely generated subgroup, and $H\rightarrow B$ be the immersed graph representing it. Then the universal cover $\widetilde{H}$ of $H$ is an $\mathcal{H}$-subtree of $\widetilde{B}$. Moreover, the $\mathcal{F}$-invariant order on $\widetilde{B}^1$ restricts to an $\mathcal{H}$-invariant order on $\widetilde{H}^1$. In particular, $\mathcal{B}\widetilde{H}\subset \mathcal{B}\widetilde{B}$.\\
Lemmas~\ref{lem:hnc}, \ref{lem:hncc}, and \ref{lem:hnccc} below are taken from \cite{MR2914871}.
\begin{lem}\label{lem:hnc}
Let $\tilde{e}$ be a bridge in an ordered $\mathcal{H}$-tree $\widetilde{H}$. Then there exists a line $L\subset \widetilde{H}$ such that $\tilde{e}$ is the unique bridge in $L$.
\end{lem}
\begin{lem}\label{lem:hncc}
Let $\mathcal{H}\leq \mathcal{F}$ be a non trivial finitely generated subgroup and $H\rightarrow B$ be the immersed graph representing it. Let $\widetilde{H}\subset \widetilde{B}$ be the universal cover of $H$. Suppose $\Gamma\subset \widetilde{H}$ is an ordered  $\mathcal{H}$-forest containing no bridges. Then $\mrank\left(\mathcal{H}\setminus \Gamma\right)=0$. In particular, $\mathcal{H}\setminus \Gamma$ has cyclic connected components.
\end{lem}
\noindent Recall that a set $\mathcal{E}\subset H^1$ is essential if no component of $H-\mathcal{E}$ is a tree. It is maximal essential if, in addition we have $\left|\mathcal{E}\right|=\mrank\left(H\right)$.
\begin{lem}\label{lem:hnccc}
Let $\mathcal{H}\leq \mathcal{F}$ be a non trivial finitely generated subgroup and $\widetilde{H}$ be an ordered $\mathcal{H}$-tree. Then $\mathcal{H}\setminus \mathcal{B}\widetilde{H}$ is a maximal essential set, hence $\left|\mathcal{H}\setminus \mathcal{B}\widetilde{H}\right|=\mrank\left(H\right)$.
\end{lem}

\noindent Immediate corollaries to Remark~\ref{rem:cardinality of fibers} and Lemmas~\ref{lem:hnc}, \ref{lem:hncc}, and \ref{lem:hnccc} are the following:
\begin{cor}\label{cor:hncccc}
If $\mathcal{H}\setminus \mathcal{B}\widetilde{H}\rightarrow B$ is injective then $H$ is inert.
\end{cor}
\noindent Note that the contrapositive to Corollary~\ref{cor:hncccc} is a statement about the space of orderings of $\mathcal{F}$, that is:
\begin{cor}\label{cor:hnccccc}
If $H$ is not inert, then there is no left-order on $\mathcal{F}$ for which $\mathcal{H}\setminus \mathcal{B}\widetilde{H}\rightarrow B$ is injective. 

\end{cor}
\subsection{Generalized Echelon Graphs}
\noindent Let $\ell:H\rightarrow B$ be an immersed graph and $\mathcal{E}=\left\{e_1,\ldots,e_{m}\right\}$ be a maximal essential set of $H$. For each $i$, let $A_i$ be the arc containing the essential edge $e_i$, and let $C_i$ be the component of $H-\displaystyle\bigcup_{j\neq i}A_j$ containing the edge $e_i$. Finally, let $H_i=\displaystyle\bigcup_{j=1}^{i}C_j$. See Figure~\ref{fig:graph_11} for an illustration. In this setting, we make the following definition:
\begin{defn}[Generalized Echelon Graphs]\label{defn:reduced echelon}
An immersed graph $\ell:H\rightarrow B$ is \textit{generalized echelon} if $H$ has a maximal essential set $\mathcal{E}=\left\{e_1,\ldots,e_{m}\right\}$ such that
\begin{enumerate}
\item $\ell|_{\mathcal{E}}$ is injective, 
\item $\ell\left(H_i\right)\cap \ell \left(\left\{ e_{i+1},\ldots,e_{m} \right\}\right)=\emptyset$.
\end{enumerate}
\begin{figure}[t]\centering
\includegraphics[width=0.6\textwidth]{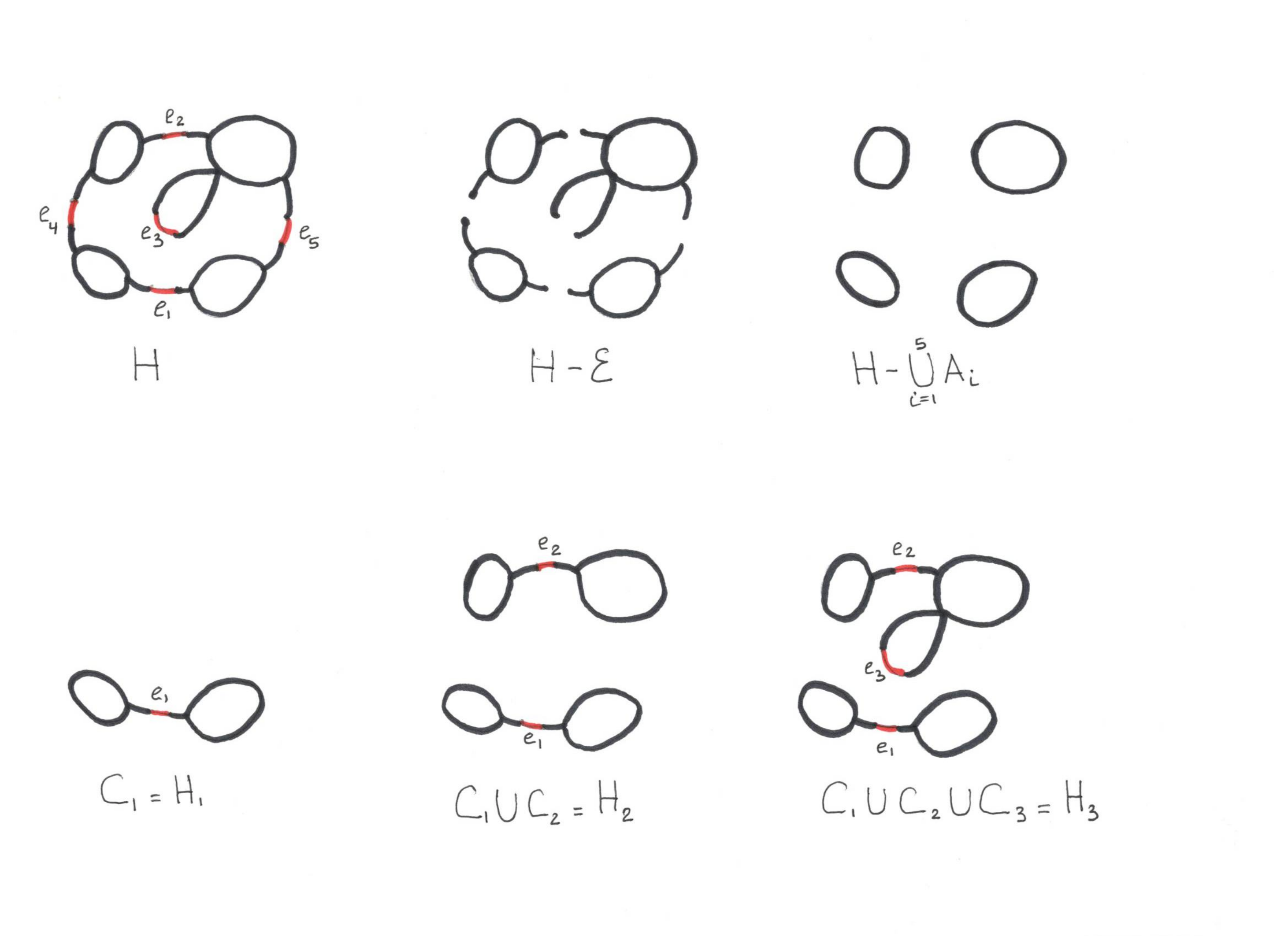}
\caption[Example of Generalized Echelon Construction]{\label{fig:graph_11}
Example of Generalized Echelon Construction}
\end{figure}
A subgroup $\mathcal{H}\leq \mathcal{F}$ is \textit{generalized echelon} if its corresponding immersion is reduced echelon.
\noindent See Example~\ref{fig:graph_1}.
\end{defn}
\begin{thm}\label{prop:echelon graph to subgroup}
Generalized echelon graphs are inert
\end{thm}
\begin{proof}
By Corollary~\ref{cor:hncccc}, to show that $\ell$ is inert, it is sufficient to find:
\begin{enumerate}
\item A maximal essential set $\mathcal{E}'$ of $H$ such that $\ell|_{\mathcal{E}'}$ is injective, 
\item An $\mathcal{F}$-invariant ordering of the universal cover $\widetilde{B}\rightarrow B$ such that the pre-image of $\mathcal{E}'$ in $\widetilde{H}\subset\widetilde{B}$ is a set of bridges. 
\end{enumerate}
To that end, let $<_{_\mathcal{F}}$ be a left-invariant order on $\mathcal{F}$ and $<_{_\mathcal{D}}$ be a total order on $\mathcal{D}$, where $\mathcal{D}$ is as defined as in Section~\ref{sec:D}. Since $\mathcal{D}$ is in one-to-one correspondence with $B^1$, the order on $\mathcal{D}$ is defined by the order on $B^1=\left\{a_1,\ldots,a_{n-m},\ell\left(e_1\right),\ldots,\ell\left(e_{m}\right)\right\}$ which we define as  
$$a_1\ <_{_\mathcal{D}}\ \ldots\  <_{_\mathcal{D}}\ a_{n-m}\ <_{_\mathcal{D}}\ \ell\left(e_1\right)\ <_{_\mathcal{D}}\ \ell\left(e_2\right)\ <_{_\mathcal{D}}\ \ldots\ <_{_\mathcal{D}}\ \ell\left(e_{m}\right)$$ 
where $a_1,\ldots,a_{n-m}$ are the edges in $B^1-\ell\left(\mathcal{E}\right)$. Note that the ordering of $a_1,\ldots,a_{n-m}$ is irrelevant as long as we require that $a_j<_{_\mathcal{D}}\ell\left(e_i\right)$ for all $e_i\in \mathcal{E}$. Then $<_{_\mathcal{D}}$ and $<_{_\mathcal{F}}$ induce a left-invariant lexicographic order $<$ on $\widetilde{B}^1\cong \mathcal{D}\times \mathcal{F}$ as described in Remark~\ref{rem:lexico}.\\
Consider the essential edge $e_i\in\mathcal{E}$ for each $i$ in increasing order. Let $\widetilde{C}_i\subset \widetilde{H}$ be the universal cover of $C_i$ and $\stab_H(\widetilde{C}_i)$ be its stabilizer, so $C_i=\stab_H(\widetilde{C}_i)\setminus \widetilde{C}_i$. By definition, $\mrank\left(C_i\right)=1$, thus by Lemma~\ref{lem:hncc}, $\widetilde{C}_i$ must contain a bridge $\tilde{e}_i'$ whose image under the quotient map is some essential edge $e_i'\in C_i$. Moreover, since $\widetilde{C}_i\rightarrow C_i$ is a covering map, the line $L_i\subset \widetilde{C}_i$ in which $\tilde{e}_i'$ is a bridge must project surjectively onto $C_i$. Indeed, if $L_i\rightarrow C_i$ is not surjective, then $\stab_H(L_i)\setminus L_i$ must have rank $<2$ making $e_i'$ non-essential. By our ordering $<$, it must be the case that $\ell\left(e_i'\right)=\ell\left(e_i\right)$, since $e_i'$ is the image a bridge, and $\ell\left(e_i\right)$ is the largest label in $C_i$.\\
Let $\mathcal{E}_i=\left\{e_1',\ldots,e_i'\right\}\cup\left\{e_{i+1},\ldots,e_m\right\}$. Then $\ell|_{\mathcal{E}_i}$ is injective. By Lemma~\ref{lem:maximal essential}, $\mathcal{E}_i$ remains maximal essential in $H$. In particular, $\mathcal{E}_m$ is maximal essential that maps injectively into $B^1$, and lifts to bridges in $\widetilde{H}$. Therefore $\ell: H\rightarrow B$ is inert.
\end{proof}
\noindent Consider the graph in Figure~\ref{fig:graph_1}.
\begin{figure}[t]\centering
\includegraphics[width=.4\textwidth]{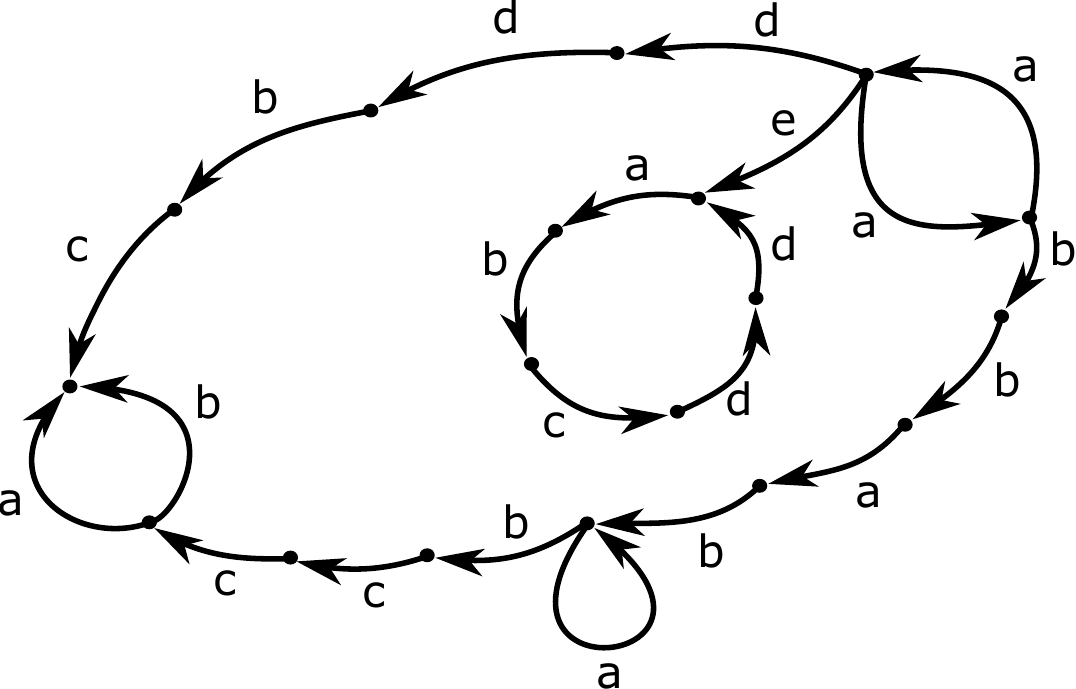}
\caption[Reduced Echelon Graph]{\label{fig:graph_1}
Reduced Echelon graph $H$}
\end{figure}
Let $\mathcal{E}=\left\{e_1,e_2,e_3,e_4\right\}$ be a maximal essential set with labels $\ell\left(\mathcal{E}\right)=\left\{b,c,d,e\right\}$ respectively, as shown in Figure~\ref{fig:graph_2}. Let $a<_{_\mathcal{D}}b<_{_\mathcal{D}}c<_{_\mathcal{D}}d<_{_\mathcal{D}}e$ be an ordering of $B^1$. Then the lifts of $e_1$ are bridges in the translates of the labelled line
 $$\ldots aaab\textbf{b}abaaa\dots$$
 Similarly, $e_2$ lifts to bridges in the translates of the labelled line 
  $$\ldots aaab\textbf{c}cab^{-1}ab^{-1}ab^{-1}\dots$$
\begin{figure}[t]\centering
\includegraphics[width=.5\textwidth]{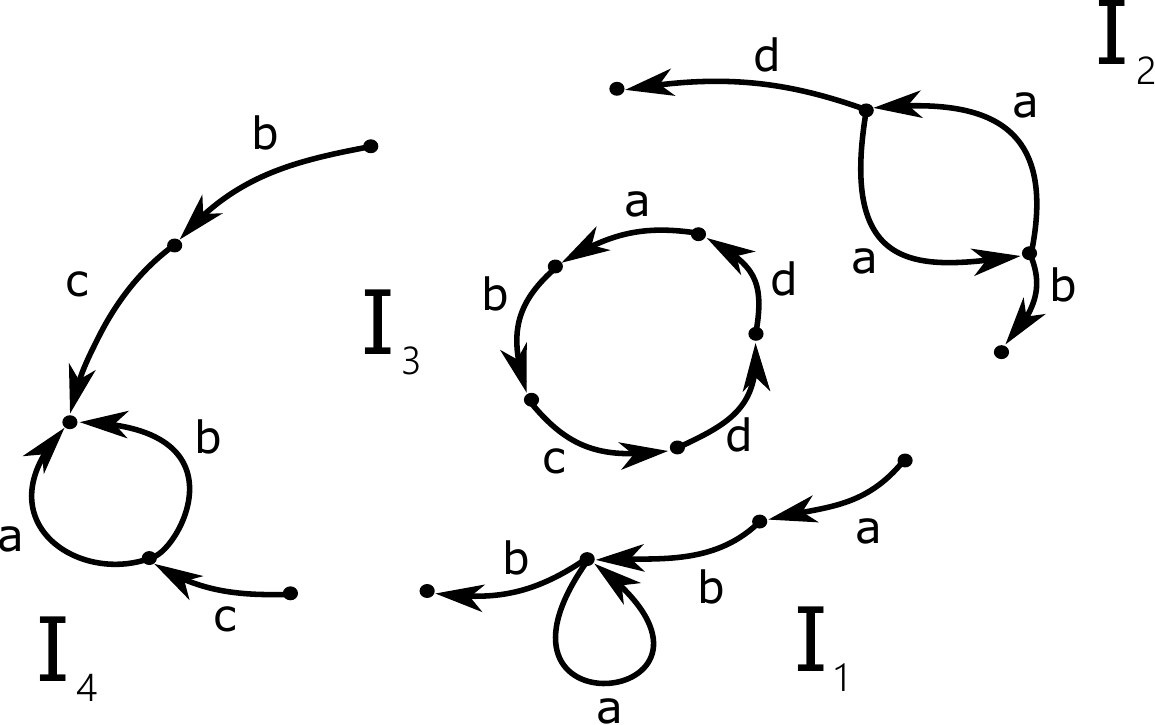}
\caption[Example of $H-\mathcal{E}$]{\label{fig:graph_2}
$H-\mathcal{E}$}
\end{figure}
\begin{cor}\label{cor:trivial}
Rank $2$ graphs (subgroups) are generalized echelon.
\end{cor}
\begin{proof}
Let $H$ be a graph of rank $2$. Then it has only one essential edge. So $|\mathcal{E}|=1$. Then $\ell|_\mathcal{E}$ is trivially injective, and $\ell\left(H_1\right)\cap\ell\left(\emptyset\right)=\emptyset$ also holds trivially. Note that in this case, we have $C_1=H_1=H$. 
\end{proof}
\begin{cor}
By Lemma~\ref{lem:inert iff}, generalized echelon subgroups are inert.
\end{cor}
\subsection{Echelon Subgroups}
\noindent The following definitions are due to Rosenmann \cite{MR3245107}:
\begin{defn}[Echelon Form]\label{defn:echelon form}
Let $X=\left\{x_1,\ldots,x_n\right\}$ be an ordered free basis of $\mathcal{F}$ and $\mathcal{H}\leq \mathcal{F}$ be a subgroup. Then $\mathcal{H}$ is in \textit{echelon form} with respect to $X$ if there exists an ordered basis of $\mathcal{H}$,  $\left<y_1,\ldots,y_m\right>$, such that each generator $y_i$ contains at least one element of $X$ that has not appeared in $y_1,\ldots,y_{i-1}$.
\end{defn}
\noindent For example if $X=\left\{a,b,c,d,e\right\}$, then $\mathcal{H}=\left<ab,a^2cb,ce\right>$ is in echelon form with respect to $X$.
\begin{defn}[Echelon Subgroup]\label{defn:echelon subgroup}
A subgroup $\mathcal{H}\leq \mathcal{F}$ is an \textit{echelon subgroup} if it is in echelon form with respect to some basis of $\mathcal{F}$.
\end{defn}
\noindent We extend this definition to graphs as follows: a graph $H$ is \textit{echelon} if $\pi_1H$ is echelon.\\
\noindent In \cite{MR3245107}, it was shown that echelon subgroups arise as images of series of particular endomorphisms $\mathcal{F}\rightarrow \mathcal{F}$, called $1$-\textit{generator} endomorphisms, which were shown to have inert images. Below, we show inertness of echelon subgroups by showing that the graphs representing them are generalized echelon, thus inert.\\
Let $\mathcal{H}=\left<y_1,\ldots,y_m\right>$ be an echelon subgroup with respect to $\mathcal{F}=\left<x_1,\ldots,x_n\right>$. Denote by $x^{i}$ a label in $y_{i}$ that does not appear in $y_j$ for all $j<i$. Let $\ell:\left(H,v\right)\rightarrow B$ be the based immersion representing $\mathcal{H}$. As in Section~\ref{sec:graphs}, let $p_i:Y_i\rightarrow \left(H,v\right)$ be the based cycle corresponding to the generator $y_i$. Denote by $e^{i}$ the edge of the cycle $p_i$ such that $\ell\left(e^{i}\right)=x^{i}$. Consider the set of edges $\mathcal{E}=\left\{e^2,\ldots,e^m\right\}$. Then we have the following lemma: 
\begin{lem}\label{lem:ech}
The set of edges $\mathcal{E}=\left\{e^2,\ldots,e^m\right\}$ is maximal essential in $H$. 
\end{lem}
\begin{proof}
Since $\mrank\left(H\right)=\left|\mathcal{E}\right|$, we only need to show that $\mathcal{E}$ is essential. By Lemma~\ref{lem:essential set}, it is sufficient to show that no component of $H-\mathcal{E}$ is a tree.\\
Suppose $T\subset \left(H-\mathcal{E}\right)$ is a tree component. Let $\mathcal{E}'=\left\{e\in \mathcal{E}\mid o\left(e\right)\in T^0\ \text{or}\ \tau\left(e\right)\in T^0\right\}$,  ie: the set of edges of $\mathcal{E}$ incident on $T$.  Let $e^s$ be the edge of $\mathcal{E}'$ with the lowest superscript. Now, the base vertex $v$ of $H$ is either in $T$ or not. Observe that $v$ cannot be in $T$, for otherwise the cycle $p_1$ would include an edge $e^k\in\mathcal{E}'$ with $k>1$ meaning that the generator $y_1$ would contains the label $x^k$ which is a contradiction. On the other hand, if $v\notin T$, then once again we have a contradiction since the cycle $p_s$ must include an edge of $\mathcal{E}'$ of higher superscript. Hence no component is a tree. 
\end{proof}
\begin{rem}\label{rem:echh}
By construction, $\ell|_\mathcal{E}$ is injective.
\end{rem}
\begin{prop}
Echelon graphs are generalized echelon.
\end{prop}
\begin{proof}
Let $\mathcal{E}=\left\{e^2,\ldots,e^m\right\}$ be the essential set from Lemma~\ref{lem:ech}. For each $2\leq i\leq m$, let $A_i$ be the arc containing the essential edge $e^{i}$, $C_i$ be the component of $H-\displaystyle\bigcup_{j\neq i}A_j$ containing $e^{i}$, and $H_i=\displaystyle\bigcup_{j=2}^{i}C_j$. Observe that in this case,  $C_i=p_i\left(Y_i\right)\cup p_1\left(Y_1\right)$, after folding. So $H_i=\displaystyle\bigcup_{j=1}^{i}p_j\left(Y_j\right)$ which is simply the union of the cycles corresponding to the generators $y_1,\ldots,y_i$. We have by construction $\ell\left(H_i\right)\cap \left\{x^{i+1},\ldots,x^m\right\}=\emptyset$. By Lemma~\ref{lem:ech}, Remark~\ref{rem:echh}, and  Definition~\ref{defn:reduced echelon}, $H$ is generalized echelon.
\end{proof}
\begin{cor}
Echelon subgroups are generalized echelon.
\end{cor}
\noindent  We now give an example of a generalized echelon subgroup that is not echelon. But first we prove the following lemma:
\begin{lem}\label{lem:commutator}
Let $\mathcal{H}\leq \mathcal{F}$ be an echelon subgroup such that $\rank\left(\mathcal{H}\right)=\rank\left(\mathcal{F}\right)=n\geq 1$, and $\Phi:\mathcal{F}\rightarrow \mathbb{Z}^n$ be the abelianization homomorphism. Then $\Phi\left(\mathcal{H}\right)\neq \left\{0\right\}$.
\end{lem}
\begin{proof}
By definition of echelon subgroups, each generator $y_i$ has a letter that does not appear in any $y_j$ for all $j<i$. Since $\rank\left(\mathcal{H}\right)=n$, it must be the case that $y_1$ contains only one letter. Thus, $\Phi\left(y_1\right)\neq 0$.
\end{proof}
\noindent Suppose $n=2$, that is $B$ is the bouquet of two circles with $\pi_1B=\mathcal{F}=\left<a,b\right>$. Consider the immersed graph $H\rightarrow B$ in Figure~\ref{fig:graph_12}. Then $\mathcal{H}=\left<aba^{-1}b^{-1},a^{-1}b^{-1}ab\right>$.
\begin{figure}[t]\centering
\includegraphics[width=.3\textwidth]{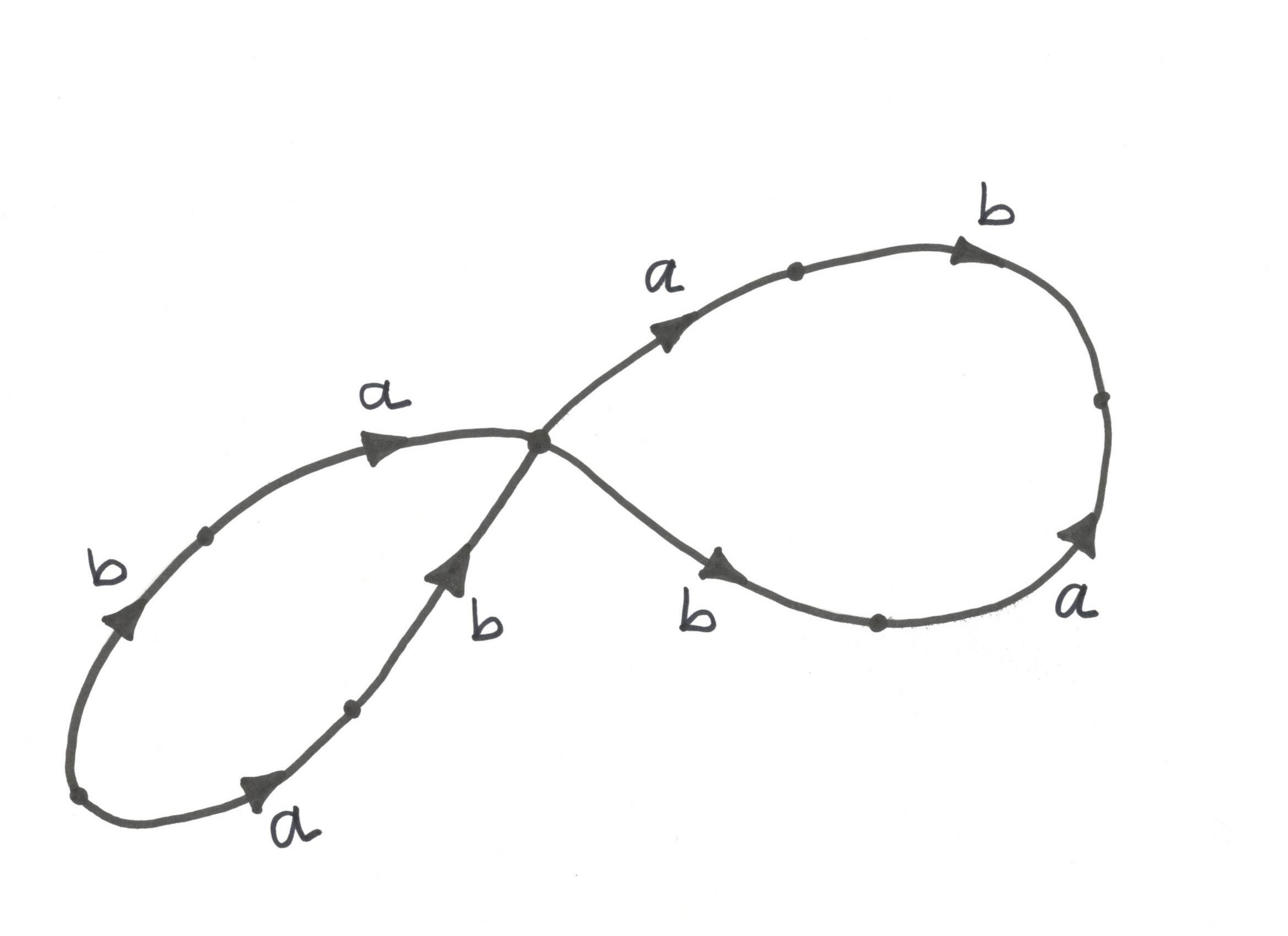}
\caption[Example of a generalized echelon graph that is not echelon.]{\label{fig:graph_12}
Example of a generalized echelon graph that is not echelon}
\end{figure}
By Corollary~\ref{cor:trivial}, $\mathcal{H}$ is generalized echelon. If $\Phi: \mathcal{F}\rightarrow \mathbb{Z}^2$ is the abelianization homomorphism, then $\Phi\left(\mathcal{H}\right)=\left\{0\right\}$, and by Lemma~\ref{lem:commutator} $\mathcal{H}$ is not echelon. In particular, for any automorphism $\Psi: \mathcal{F}\rightarrow\mathcal{F}$, we have $\Phi\circ\Psi\left(\mathcal{H}\right)=\left\{0\right\}$. That is, for any sequence of Nielsen transformations one can perform of the basis of $\mathcal{H}$, the resulting basis will always have its elements containing both letters $a$ and $b$. Therefore $\mathcal{H}$ cannot be put in echelon form.
\section{Further Questions}
\noindent We propose the following questions for further research:
\begin{enumerate}
\item In Section~\ref{sec:inert and compressed} we defined inertness of a set of vertices and showed that it implies inertness of the graph. We believe that a more general property can be defined for  arcs, and doing so will allow us to generate an even larger class of examples.
\item Do immersed bouquets of circles have injective sets of bridges? 
\end{enumerate}
%

\bibliographystyle{alpha}
\bibliography{brahim.bib}

\def\cprime{$'$} \def\polhk#1{\setbox0=\hbox{#1}{\ooalign{\hidewidth
  \lower1.5ex\hbox{`}\hidewidth\crcr\unhbox0}}} \def\cprime{$'$}
  \def\cprime{$'$} \def\polhk#1{\setbox0=\hbox{#1}{\ooalign{\hidewidth
  \lower1.5ex\hbox{`}\hidewidth\crcr\unhbox0}}}
\begin{thebibliography}{MKS76}

\bibitem[DNR14]{https://doi.org/10.48550/arxiv.1408.5805}
B.~Deroin, A.~Navas, and C.~Rivas.
\newblock Groups, orders, and dynamics, 2014.

\bibitem[DS75]{MR369529}
Joan~L. Dyer and G.~Peter Scott.
\newblock Periodic automorphisms of free groups.
\newblock {\em Comm. Algebra}, 3:195--201, 1975.

\bibitem[DV96]{MR1385923}
Warren Dicks and Enric Ventura.
\newblock {\em The group fixed by a family of injective endomorphisms of a free
  group}, volume 195 of {\em Contemporary Mathematics}.
\newblock American Mathematical Society, Providence, RI, 1996.

\bibitem[Fri15]{MR3289057}
Joel Friedman.
\newblock Sheaves on graphs, their homological invariants, and a proof of the
  {H}anna {N}eumann conjecture: with an appendix by {W}arren {D}icks.
\newblock {\em Mem. Amer. Math. Soc.}, 233(1100):xii+106, 2015.
\newblock With an appendix by Warren Dicks.

\bibitem[How54]{Howson54}
A.~G. Howson.
\newblock On the intersection of finitely generated free groups.
\newblock {\em J. London Math. Soc.}, 29:428--434, 1954.

\bibitem[Joh80]{MR695161}
D.~L. Johnson.
\newblock {\em Topics in the theory of group presentations}, volume~42 of {\em
  London Mathematical Society Lecture Note Series}.
\newblock Cambridge University Press, Cambridge-New York, 1980.

\bibitem[KM02]{MR1882114}
Ilya Kapovich and Alexei Myasnikov.
\newblock Stallings foldings and subgroups of free groups.
\newblock {\em J. Algebra}, 248(2):608--668, 2002.

\bibitem[Min12]{MR2914871}
Igor Mineyev.
\newblock Groups, graphs, and the {H}anna {N}eumann conjecture.
\newblock {\em J. Topol. Anal.}, 4(1):1--12, 2012.

\bibitem[MKS76]{MR0422434}
Wilhelm Magnus, Abraham Karrass, and Donald Solitar.
\newblock {\em Combinatorial group theory}.
\newblock Dover Publications, Inc., New York, revised edition, 1976.
\newblock Presentations of groups in terms of generators and relations.

\bibitem[MV04]{MR2097438}
A.~Martino and E.~Ventura.
\newblock Fixed subgroups are compressed in free groups.
\newblock {\em Comm. Algebra}, 32(10):3921--3935, 2004.

\bibitem[Neu57]{MR93537}
Hanna Neumann.
\newblock On the intersection of finitely generated free groups. {A}ddendum.
\newblock {\em Publ. Math. Debrecen}, 5:128, 1957.

\bibitem[Neu90]{MR1092229}
Walter~D. Neumann.
\newblock On intersections of finitely generated subgroups of free groups.
\newblock In {\em Groups---{C}anberra 1989}, volume 1456 of {\em Lecture Notes
  in Math.}, pages 161--170. Springer, Berlin, 1990.

\bibitem[Ros13]{MR3245107}
Amnon Rosenmann.
\newblock On the intersection of subgroups in free groups: echelon subgroups
  are inert.
\newblock {\em Groups Complex. Cryptol.}, 5(2):211--221, 2013.

\bibitem[Sta83]{MR695906}
John~R. Stallings.
\newblock Topology of finite graphs.
\newblock {\em Invent. Math.}, 71(3):551--565, 1983.

\end{thebibliography}

\end{document}